\documentclass[11pt, a4paper]{article}
\usepackage{amsmath,amsthm,amssymb,amsfonts}
\usepackage{appendix}
\usepackage{fullpage}
\usepackage{french}
\newtheorem{theorem}{Theorem}[section]
\newtheorem{lemma}[theorem]{Lemma}
\newtheorem{proposition}[theorem]{Proposition}

\newtheorem{corollary}[theorem]{Corollary}

\theoremstyle{remark}

\newtheorem{remark}[theorem]{Remark}

\def\Xint#1{\mathchoice
{\XXint\displaystyle\textstyle{#1}}%
{\XXint\textstyle\scriptstyle{#1}}%
{\XXint\scriptstyle\scriptscriptstyle{#1}}%
{\XXint\scriptscriptstyle\scriptscriptstyle{#1}}%
\!\int}
\def\XXint#1#2#3{{\setbox0=\hbox{$#1{#2#3}{\int}$ }
\vcenter{\hbox{$#2#3$ }}\kern-.6\wd0}}

\def\dashint{\Xint-}

\usepackage{graphicx}
\usepackage[usenames,dvipsnames]{color}
\usepackage[colorlinks=true, pdfstartview=FitV, linkcolor=black, citecolor=blue, urlcolor=blue]{hyperref}

\makeatletter

\@addtoreset{equation}{section}
\makeatother

\newcommand{\E}{{\mathbb E}}

\newcommand{\N}{{\mathbb N}}

\newcommand{\R}{{\mathbb R}}

\newcommand{\mO}{\mathcal{O}}

\newcommand{\mI}{\mathcal{I}}
\newcommand{\mM}{\mathcal{M}}
\newcommand{\mN}{\mathcal{N}}

\newcommand{\pa}{{\partial}}
\newcommand{\na}{{\nabla}}

\newcommand{\eps}{{\varepsilon}}

\def\div{\hbox{\rm div \!}}

\def\sspace{\smallskip \noindent}
\def\mspace{\medskip \noindent}
\def\bspace{\bigskip\noindent}

\title{Derivation of the Batchelor-Green formula \\ for random suspensions} 
\author{David G\'erard-Varet}

\begin{document}
\maketitle

\begin{resume} Cet article est consacré à la détermination de la viscosité effective d'une suspension sans inertie, à faible fraction volumique solide $\phi$. Le but est d'en obtenir rigoureusement une approximation explicite à l'ordre $\phi^2$. Dans les articles \cite{GVH, GVM}, une telle approximation a été obtenue dans le cas de sphères solides satisfaisant une hypothèse forte de séparation $ d_{min} \ge c \phi^{-\frac13} r$, avec $d_{min}$ la distance minimale entre les sphères et $r$ leur rayon. Des formules très explicites ont été fournies pour le cas de distributions périodiques ou aléatoires stationnaires (avec séparation forte) de particules. Nous considérons ici une autre classe de distributions aléatoires (sans doute plus réaliste), sous des hypothèses de séparation beaucoup moins fortes et des hypothèses de décorrélation à l'infini. Nous justifions en particulier dans ce contexte la célèbre formule de Batchelor-Green \cite{BG}. Notre résultat s'applique par exemple au cas où les boules sont distribuées selon un processus de Poisson {\em hardcore}, satisfaisant l'hypothèse presque minimale $d_{min} > (2+\eps) r$, $\eps > 0$.  

\medskip
{\em Mots clés : viscosité effective, homogénéisation, processus ponctuels, équations de Stokes. }
\end{resume}

\english

\abstract{This paper is dedicated to the effective viscosity of suspensions without inertia, at low solid volume fraction $\phi$. The goal is to derive rigorously a $o(\phi^2)$ formula for the effective viscosity.  In \cite{GVH, GVM}, such formula was given for rigid spheres  satisfying the strong separation assumption $ d_{min} \ge c \phi^{-\frac13} r$, where $d_{min}$ is the minimal distance between the spheres and $r$ their radius. It was then applied to both periodic and random configurations with separation, to yield explicit values for the $O(\phi^2)$ coefficient. We consider here complementary (and certainly more realistic)  random configurations, satisfying softer assumptions of separation, and long range decorrelation. We justify in this setting the famous Batchelor-Green formula  \cite{BG}. Our result applies for instance to hardcore Poisson point process with almost minimal hardcore assumption $d_{min} > (2+\eps) r$, $\eps > 0$. 

}

\section{Introduction}
The most basic model to study the effective viscosity created by a suspension of rigid balls in a fluid is the following. 
Given the family  $(B_i)_{i \in I}$ of the balls,  indexed by a finite subset $I$ of $\N$,  and given  $S \in \text{Sym}_{3,\sigma}(\mathbb{R})$, the space of 3x3 symmetric and trace-free matrices, we consider the system
\begin{equation} \label{stokes}
\begin{aligned}
-\Delta u_{I,S}  + \na p_{I,S}  & = 0,  \quad  x \in \R^3 \setminus (\cup_{i \in I} B_i), \\
\div u_{I,S}   & = 0,  \quad  x \in \R^3  \setminus (\cup_{i \in I} B_i), \\
D u_{I,S} & = 0,  \quad  x \in \cup_{i \in I} B_i, \\
\int_{\pa B_i} \sigma(u_{I,S}, p_{I,S}) n   &  = 0,  \quad  \forall i \in I, \\
\int_{\pa B_i} \sigma(u_{I,S}, p_{I,S}) n \times (x-x_i) & = 0,   \quad  \forall i \in I, \\
\lim_{ |x| \rightarrow +\infty} u_{I,S} - S x & = 0. 
\end{aligned}
\end{equation}
Here, $D = \frac{1}{2} (\na + \na^t)$ denotes the symmetric gradient, $ \sigma(u,p) := 2 D(u) - pI$ denotes the newtonian stress tensor, and $x_i$ is the center of $B_i$. The first two lines correspond to a steady Stokes flow outside the balls, meaning that we neglect the inertia of the fluid. The third line expresses that the velocity field in each $B_i$ is rigid:  it is equivalent to $u = u_i + \omega_i \times (x-x_i)$ on $B_i$ for some constant vectors $u_i, \omega_i \in \R^3$. The fourth, resp.  fifth line, expresses that the force, resp. the torque on each particle is zero. In particular, we neglect  the gravitational force ({\em neutral buoyancy}). Finally, the last condition models the application of a strain on the fluid. As soon as closed balls $B_i$'s are disjoint, now standard arguments yield existence and uniqueness of a solution 
$u_{I,S} \in \dot{H}^1(\R^3) := \{ u \in L^6(\R^3), \quad \na u \in L^2(\R^3)\}$.

\mspace
It is well-known that the rigidity of the balls creates resistance to strain. If the number of balls is large, one may expect some averaging to take place, so that this extra resistance  may be interpreted as an extra viscosity for the fluid.   The hope is to replace the fluid-particles system by a simpler Stokes equation in the whole of $\R^3$, with an {\em effective viscosity tensor} in the region where the particles stand.

\mspace
A path towards the derivation of  an effective model is homogenization theory.  Typically, one considers balls $B_{i, \eps} = B(x_{i, \eps}, \eps)$, where $\eps \ll 1$ is the radius of the balls, and where the centers $x_{i, \eps}$ are built from a periodic or random stationary distribution of points. In the context of suspensions, the random modeling is of course far more relevant. The balls occupy a macroscopic domain $\mO$  of typical size $O(1)$, and one considers the asymptotics of $v_\eps$ solution of \eqref{stokes} with $B_i = B_{i, \eps}$, 
as $\eps \rightarrow 0$. The limit is the solution of a Stokes equation over $\R^3$ with non-constant viscosity tensor $\mu_{\mathrm{eff}}$: one expects $\mu_{\mathrm{eff}} = 1$ outside $\mO$, $\mu_{\mathrm{eff}} = \mu_h$ a homogenized viscosity tensor inside $\mO$. While such result has been known for long in the context of the Laplace equation \cite[chapter 3]{book:ZKO}, to our knowledge, the Stokes case was only analyzed  in the recent paper \cite{DuerinckxGloria}. See \cite{BelKoz,BasGer,GerVa,DalGer,GiuHo,DueGlo} for other stochastic homogenization results in fluid mechanics. In \cite{DuerinckxGloria}, the authors consider a stationary and ergodic point process $(x_i)_{i \in \N}$ over $\R^3$, and  $B_{i,\eps} := \eps B_i$, where $B_i := \overline{B(x_i, 1)}$ is the closed unit ball centered at $x_{i}$.  The point process satisfies the assumption 
\begin{equation} \label{H1} \tag{H1}
\exists R_0 > 2, \quad  \inf_{i \neq j} |x_i - x_j| \ge R_0\quad \text{almost surely}. 
\end{equation}
 which is a slight reinforcement of the natural non-penetration condition of the particles. As shown in \cite{DuerinckxGloria}, the effective viscosity tensor is then classically expressed in terms of a corrector problem without small parameter, set in $\R^3 \setminus (\cup_i B_i)$. The corrector equations are  
 \begin{equation} \label{corrector}
\begin{aligned}
 -\Delta \Phi_S  + \na P_S   & =  0, \quad  x \in \R^3 \setminus (\cup_{i \in \N} B_i), \\
 \div \Phi_S  & =  0, \quad  x \in \R^3  \setminus (\cup_{i \in \N} B_i), \\
D \Phi_S  + S & =  0, \quad  x \in \cup_{i \in \N} B_i, \\
\int_{\pa B_i} \sigma(\Phi_S+S, P_S) n  & =  0, \quad  \forall i \in \N, \\ 
\int_{\pa B_i} \sigma(\Phi_S+S, P_S) n \times (x-x_i)  & =  0, \quad  \forall i \in \N. 
\end{aligned}
\end{equation}
The unique solvability of the corrector equation is given by 
\begin{proposition} \label{prop_corrector} {\bf \cite[Proposition 2.1]{DuerinckxGloria}}

\sspace
Assume that $(x_i)_{i \in \N}$ is given by an ergodic stationary point process satisfying the separation condition \eqref{H1}. 
Then,  there exists a unique random field $(\Phi_S, P_S)$ such that 
\begin{itemize}
\item Almost surely, $(\Phi_S, P_S)$ is a solution of \eqref{corrector} in $H^1_{loc}(\R^3) \times L^2_{loc}(\R^3\setminus \cup B_i)$, 
\item $\na \Phi_S, \: P_S 1_{\R^3 \setminus \cup B_i}$ are stationary,  
\item $\E \na \Phi_S = 0$, $\quad \E P_S 1_{\R^3 \setminus \cup B_i} = 0$, $\quad \int_{\cup B_i} \Phi_S = 0$ almost surely. 
\item $\E \,  |\na \Phi_S|^2 + \E \, \big( p_S^2 1_{\R^3 \setminus \cup B_i} \big)  < +\infty$ 
\end{itemize}
\end{proposition}

\mspace
From there, one can define  the effective or homogenized viscosity tensor $\mu_h$ in terms of the solutions $\Phi_S$, $S$ varying in  $\text{Sym}_{3,\sigma}(\mathbb{R})$. It is the element of  $\text{Sym}\left(\text{Sym}_{3,\sigma}(\mathbb{R})\right)$ given by: 
\begin{equation} \label{effective_viscosity}
\forall S \in \text{Sym}_{3,\sigma}(\mathbb{R}), \quad \mu_h S : S \: := \:   \E |D(\Phi_S + S)|^2
\end{equation}
The relevance of $\mu_h$ to the homogenization problem is asserted by the following proposition. For later notational convenience, we denote $\eps = N^{-\frac13}$, $N \gg 1$.  
\begin{proposition} {\bf (adapted from \cite[Theorem 1]{DuerinckxGloria})} \label{thm_homogenization}

\sspace
Let $(x_i)_{i \in \N}$  a stationary ergodic point process satisfying \eqref{H1}, $B_i := \overline{B(x_i, 1)}$, $i \in \N$. Let $\mO$ any  smooth bounded domain, and 
$$I_N := \{i, \: B_i \subset  N^{\frac13} \mO\}, \quad v_{N,S} := N^{-1/3} u_{I_N,S}(N^{\frac13} \cdot), \quad u_{I_N,S} \: \text{ solution of \eqref{stokes} with } I = I_{N}. $$
Then, almost surely, $v_{N,S}$ goes weakly in $\dot{H}^1(\R^3)$ to $v_{\mathrm{eff},S}$ satisfying 
\begin{equation} \label{stokes_eff}
\begin{aligned}
- \div \Big( 2 \mu_{\mathrm{eff}} D v_{\mathrm{eff},S} \Big)  + \na q_{\mathrm{eff},S} & =  0, \quad  x \in \R^3  \\
\nonumber \div   v_{\mathrm{eff},S} & =  0, \quad  x \in \R^3  \\
\nonumber \lim_{ |x| \rightarrow +\infty} v_{\mathrm{eff},S} - S x & =  0, 
\end{aligned}
\end{equation}
where $\mu_{\mathrm{eff}} = \mu_{\mathrm{eff}}(x) \in \text{Sym}\left(\text{Sym}_{3,\sigma}(\mathbb{R})\right)$ is defined by: 
\begin{equation*}
 \mu_{\mathrm{eff}}(x) S : S  = |S|^2, \quad x \in \mO^c, \quad  \mu_{\mathrm{eff}}(x) S : S   = \mu_h S : S, \quad x \in \mO, 
 \end{equation*}
where $\mu_h$ is defined in \eqref{effective_viscosity}.
\end{proposition}

\mspace
Strictly speaking, only the case of a bounded domain $\mO$ with Dirichlet condition is considered in \cite{DuerinckxGloria}, but the case considered here (extension by the homogeneous Stokes solution outside $\mO$) can be covered as well with minor modifications. See also the proof of Proposition  \ref{approx_effective_viscosity} in Appendix \ref{appA}. 
 
\mspace
Still, the numerical approximation  of the solution of \eqref{corrector} is very demanding, so that such kind of effective model remains in practice hard to implement. To overcome this issue, physicists try to obtain simplified models in particular subregimes, notably in the dilute case, that is at small solid volume fraction $\phi$. This is all the more relevant here that system \eqref{stokes} seems to be appropriate only for $\phi \lesssim 0.2$ : beyond this typical value, it is acknowledged that frictional interactions play a substantial role on top of hydrodynamic interactions \cite{GuaPou}.  
 The goal is then to obtain effective {\em approximate} models, meaning with an error $o(\phi^\alpha)$ for some positive $\alpha$, rather than effective {\em limit} models.  To obtain such approximate models, one may try to find an expansion in powers of $\phi$ of the homogenized tensor $\mu_h$. Or one can try to bypass homogenization theory, by exhibiting directly a $o(\phi^\alpha)$ approximation of the solution $v_{N,S}$ in Proposition \ref{thm_homogenization}. 

\mspace
A famous first step in this direction was made by Einstein \cite{Einstein}. He showed that if the suspension is homogeneous, and if the interaction between the particles can be neglected, $o(\phi)$ approximation is given by $\mu_{Ein} = 1 + \frac{5}{2} \phi$. Many works, especially over the last two years, have been devoted to the justification of this claim, trying to identify milder and milder geometrical assumptions on the particles configuration under which particles interaction is indeed neglectible  \cite{MR813656,MR813657,Haines&Mazzucato,MR4102716,MR4098775}. To our knowledge, justification of Einstein's formula under the current mildest requirements is found in \cite{GVRH}.  Let us note that this work is of a deterministic nature:  it does not use the existence of an effective viscosity, or some specific random structure of the set of particles. This is consistent with the fact that Einstein's formula is about neglecting the interaction of small scale particles, while homogenization theory is about understanding the macroscopic  effect of such small scale interaction. 

\mspace
The derivation of a $o(\phi^2)$ approximation turns out to be more tricky. A rigorous treatment was carried in the recent paper \cite{GVH}, and further refined in \cite{GVM}, under a strong assumption on the minimal distance between the particles. In the context of the random process $(x_i)_{i \in \N}$ introduced here,  it corresponds to saying that the particles $x_i$ inside $N^{\frac13} \mO$ satisfy 
\begin{equation} \label{Hstrong}
 \exists c > 0, \quad  \inf_{i \neq j} |x_i - x_j| \ge c \phi^{-\frac13} \quad \text{almost surely} 
 \end{equation}
where $\phi = \E 1_{\cup_i B_i}$ is the solid volume fraction. Obviously, this assumption is  more stringent than \eqref{H1}. Its main advantage is that, as particles are far from one another, one can solve systems of the form \eqref{stokes} thanks to the so-called method of reflections, {\it cf} \cite{HV}. This allows to rely on more explicit formulas, and to show the existence of a $o(\phi^2)$ effective approximate viscosity in the form 
\begin{equation} \label{formula_muapp}
 \mu_{app} = 1 +  \frac{5}{2} \phi + \Big(  \lim_{N \rightarrow +\infty} \frac{1}{N^2} \sum_{x_i \neq x_j \in N^{\frac13} \mO} \mM(x_i - x_j) - \int_{\mO \times \mO} \mM(x-y) dx dy \Big) \phi^2
\end{equation}
for $\mM$ an explicit function homogeneous of order $-3$.  

\mspace
Let us point out that, as in \cite{GVRH}, part of the work carried in \cite{GVH, GVM} is not related to homogenization. More precisely, one does not impose  {\it a priori}   some specific random stationary or periodic structure on the points $x_i$:  roughly, as shown in \cite{GVM},  a necessary and sufficient condition for the existence of a $o(\phi^2)$ effective approximate viscosity is the existence of the mean field limit at the r.h.s. of \eqref{formula_muapp}, as well as local versions of it. Still, contrary to Einstein's formula,  this mean field limit involves small scale interactions of the point process. Hence, restricting to the usual homogenization setting is legitimate. Under such setting (periodic or random), following recent progress in the analysis of Coulomb gases \cite{MR3309890}, we were able to give an explicit formula for this mean field limit, and to justify  in this way results from the physics literature (while discarding others). For instance, in the random case, when the $2$-point correlation function $\rho_2$ of the process is radial and converges fast enough to $\phi^2$, one can show that 
$$ \mu_{app} = 1 + \frac{5}{2} \phi + \frac{5}{2} \phi^2. $$

\mspace
{\em The goal of the present paper is to go beyond  this analysis, and derive a $o(\phi^2)$ formula without the strong  assumption \eqref{Hstrong}.}  This problem was tackled at the formal level in a celebrated paper of Batchelor and Green \cite{BG}. See also \cite{AB, Hin}. In a first part, the authors derive a formula for the second order correction, involving the two point correlation function of the process $\rho_2$ and the field $\Phi_{x,y}$ that solves the Stokes problem \eqref{stokes} in the case of two balls centered at $x$ and $y$. In a second part, based on the companion paper \cite{BG2}, they discuss the possible form of the two-point correlation function of the process, and deduce approximate numerical values for the second order correction. For instance, in the case where the particles are driven by a strain at infinity like in \eqref{stokes}, they derive the formula 
$$ \mu_{app} = 1 + \frac{5}{2} \phi + 7.6 \phi^2. $$

\mspace
We focus here on the first part of \cite{BG}: our purpose is to recover rigorously a $o(\phi^2)$ approximation in terms of  $\rho_2$ and $\Phi_{x,y}$. We shall notably clarify the loose decorrelation arguments  used in \cite{BG},  and bypass the so-called renormalization technique used there. Our strategy, as well as the precise statement of our results will be given in the next section. Comparison to the recent paper \cite{DueGloria} on the same topic will also be made. Let us stress right away that the present paper is not a simple technical improvement of  \cite{GVH, GVM}. Our relaxation of \eqref{Hstrong} allows to cover a totally different class of random distribution of particles, like those associated with thinned Poisson point processes. This corresponds to a different kind of diluteness mechanism, that is more realistic. Indeed, one must keep in mind that model \eqref{stokes} is a snapshot of the suspension at a given time $t$. To analyse the time evolution of the suspension, one should consider in a second step the differential equations $\dot{x_i} = u_i$, with $u_i$ the translational velocity of ball $B_i$. In particular, it is very unlikely that the strong condition \eqref{Hstrong} on the inter-particle distance be preserved through time, see \cite{BG2}. Hence, going from \eqref{Hstrong} to  \eqref{H1}, even though it is probably not enough, is an important effort in order to capture the good dynamics of dilute suspensions. Furthermore, the mathematical approach used in  \cite{GVH, GVM}, based on the method of reflections, collapses in the present context, so that we need several new tools, like the cluster expansions described below. The main point is to handle short range correlations, which are not present in our previous studies. As a result, the formula that we derive for the effective viscosity differs from the one in  \cite{GVH, GVM}.

\section{Explanation and statement of the results}
\subsection{Cluster expansions} \label{subsec_cluster}
Our goal is to provide a second order expansion in $\phi$ of the effective viscosity tensor $\mu_h$ defined in \eqref{effective_viscosity}. We recall that $(x_i)_{i \in \N}$ is an ergodic stationary process,  that $\displaystyle B_i := \overline{B(x_i, 1)}$ and that  
\begin{equation}
 \phi := \E 1_{\cup_{i \in \N} B_i}.
\end{equation}
 Our starting point is the following 
 \begin{proposition}  \label{approx_effective_viscosity}
Let $(x_i)_{i \in \N}$ an ergodic stationary point process satisfying \eqref{H1}, and 

\sspace
$\displaystyle I_N := \{i, \: B_i \subset B(0,N^{\frac13})\}$. Then, 
\begin{align*}
 \mu_h S : S  & =  |S|^2 +   \lim_{N \rightarrow +\infty} \E  \frac{1}{ 2 |B(0,N^{\frac13})|} \sum_{i \in I_N}  \int_{\pa B_i} \big( \sigma(u_{I_N,S}, p_{I_N,S})n - 2 u_{I_N,S}  \big)   \cdot Sn 
 \end{align*}
\end{proposition}

\mspace
This result is commented and proved in Appendix \ref{appA}.

\mspace
 The idea is to obtain an expansion of $\mu_h S : S$ through a so-called {\em cluster expansion} of $u_{I_N,S}$.   Such {\em cluster expansions} were introduced by physicists in various problems from statistical or solid physics. A good reference for this is the paper \cite{Feld82}. This kind of expansion was used recently in \cite{MR3458165}  in the context of elliptic homogenization: the authors proved that when the diffusion matrix is built through Bernoulli perturbations of parameter $p$, the homogenized matrix depends analytically on $p$. The notion of cluster expansion  is  based on the following remark.  Let $J$ a subset of $\N$, and  $F = F(I)$ defined on the collection $\mathcal{P}_J$ of  finite subsets  $I \subset J$.  We define $G$ on $\mathcal{P}_J$  by the following relation: 
\begin{equation}
F(I) = \sum_{I' \subset I} G(I'), \quad \forall I \in  \mathcal{P}_J. 
\end{equation}
It is easily seen that this relation defines $G$, with the first relations: 
\begin{align*}
\# I = 0 &: \quad G(\emptyset) = F(\emptyset), \\
\# I = 1 &: \quad  G(\{i\}) = F(\{i\}) - G(\emptyset) = F(\{i\}) - F(\emptyset) \\
\# I = 2 &: \quad G(\{i,j\}) =  F(\{i,j\}) - G(\{i\})  - G(\{j\}) - G(\emptyset) = F(\{i,j\}) - F(\{i\})  - F(\{j\}) + F(\emptyset) 
\end{align*} 
More generally, one can prove by induction on $\#I$  that 
$$ G(I) = \sum_{I' \subset I} (-1)^{\#I - \#I'} F(I'). $$ 

\mspace
Inspired by this kind of expansion, we look for an approximation 
\begin{equation} \label{approx_u_NS}
\begin{aligned}
 u_{I_N,S} & = u_{\emptyset, S} + \sum_{\{k\} \subset I_N  } u_{\{k\}, S} - u_{\emptyset, S} +  \sum_{\substack{\{k,l\} \subset I_N,\\ k \neq l }} \big(  u_{\{k,l\}, S} - u_{\{k\},S}  - u_{\{l\},S} + u_{\emptyset, S} \big) + \dots  \\
 & =  Sx  + \sum_{\{k\} \subset I_N  } \Phi_{\{k\}}  +  \sum_{\substack{\{k,l\} \subset I_N,\\ k \neq l }} \big(  \Phi_{\{k,l\}} - \Phi_{\{k\}}  - \Phi_{\{l\}}  \big) + \dots  
\end{aligned}
\end{equation}
where, omitting the dependency in $S$ for brevity, we denoted 
\begin{equation} \label{def_Phi_I} 
 \Phi_I := u_{I,S} - u_{\emptyset, S} = u_{I,S} - Sx, \quad P_I := p_{I,S}. 
\end{equation} 

\mspace
The hope is that in the dilute regime where the volume fraction $\phi$ of the suspension is small, the successive terms in this approximation will contribute to successive powers of $\phi$ in a possible expansion of the effective viscosity. 

\subsection{Formal expansion of the effective viscosity} \label{subsec_formal}
Let 
\begin{equation} \label{def_I_i}
 \mI_i(u,p) := \int_{\pa B_i} \big(  \sigma(u,p)n - 2 u \big)  \cdot Sn. 
 \end{equation}
 An important property to notice  is that for any smooth $(u,p)$ on $B_i$, 
\begin{equation} \label{property_Ii}
 (u,p) \:\: \text{solution of homogeneous Stokes equations in a vicinity of } B_i \: \Rightarrow \:  \mI_i(u,p)  = 0. 
 \end{equation}  
Indeed, by elliptic regularity, $u,p$ are smooth near $B_i$, and 
 $$ \mI_i(u,p)  =  \int_{B_i} \div \sigma(u,p) : S(x-x_i) + \int_{B_i} 2 D(u) : S -  \int_{\pa B_i} 2 u \cdot Sn = 0. $$
We plug the formal approximation \eqref{approx_u_NS} in Proposition  \ref{approx_effective_viscosity} to find  
\begin{equation}
\begin{aligned}
& \mu_h S : S   = |S|^2  \\
& +  \lim_{N \rightarrow +\infty}\E \frac{1}{ 2 |B(0,N^{\frac13})|}  \sum_{i \in I_N}  \bigg( & \mI_i(u_\emptyset,0)   
 +  \sum_{k \in I_N} \mI_i(\Phi_{\{k\}}, P_{\{k\}})   
 +  \sum_{\substack{\{k, l\} \subset I_N \\ k \neq l}}  \mI_i(\Psi_{\{k,l\}}, P_{\{k,l\}}) \bigg) \\ 
& + \dots
\end{aligned}
\end{equation}
where 
\begin{equation}  \label{def_Psi_kl}
\Psi_{k,l} := \Phi_{\{k,l\}} - \Phi_{\{k\}}  - \Phi_{\{l\}}
\end{equation}
Clearly, by property \eqref{property_Ii}, $\mI_i(u_\emptyset,0) = 0$, so that
\begin{align*}
&\lim_{N \rightarrow +\infty} \E \frac{1}{ 2 |B(0,N^{\frac13})|} \sum_{i \in I_N}  \mI_i(u_\emptyset,0) =  0.   
\end{align*}
For the next term, we notice that $\Phi_{\{k\}}$ is  explicit:  
$$\Phi_{\{k\}} = \Phi_0(x-x_k), \quad  P_{\{k\}} = P_0(x-x_k)$$ 
with
\begin{equation} \label{def_Phi0}
 \Phi_0(x) = -\frac{5}{2} S : (x \otimes x) \frac{x}{|x|^5} - \frac{Sx}{|x|^5} + \frac{5}{2} (S : x \otimes x) \frac{x}{|x|^7}, \quad P_0(x) = -5 \frac{S : (x \otimes x)}{|x|^5}.
 \end{equation}
A tedious calculation yields, denoting $B_0 := \overline{B(0,1)}$: 
$$ D \Phi_0\vert_{\pa B_0}(x) = - 5 (S : x \otimes x)   x \otimes x + \frac{5}{2} \Big( Sx  \otimes x + x \otimes Sx \Big) - S $$
so that 
\begin{equation} \label{stress_tensor_Phi_0}
\sigma\big(\Phi_{0} , P_{0}\big) n =  \sigma\big(\Phi_{0} , P_{0}\big) x =  3 Sx    \quad \text{at } \: \pa B_0.
\end{equation}
and eventually 
\begin{equation} \label{stress_tensor_Phi_0_bis} 
\mI_i(\Phi_{\{i\}}, P_{\{i\}})  =    \int_{\pa B(0,1)} (\sigma(\Phi_0, P_0) n -  2 \Phi_0) \cdot Sn  = \frac{20 \pi}{3} |S|^2.
\end{equation}
Furthermore, by \eqref{property_Ii}, for all $k \neq i$, 
$$ \mI_i(\Phi_{\{k\}}, P_{\{k\}}) = 0. $$
As $\E \, \frac{\sharp I_N}{N} \rightarrow \phi$, we end up with
\begin{align*}
 \mu_h S : S & = |S|^2 + \frac{5}{2}\phi |S|^2 
   +   \lim_{N \rightarrow +\infty}   \E  \frac{1}{ 2 |B(0,N^{\frac13})|} \sum_{i \in I_N}   \sum_{\substack{\{k, l\} \subset I_N \\ k \neq l}}  \mI_i(\Psi_{\{k,l\}}, P_{\{k,l\}}) + \dots 
\end{align*}
The last expectation can be further decomposed into 
\begin{align*}
&  \E  \frac{1}{ 2 |B(0,N^{\frac13})|} \sum_{i \in I_N} \sum_{\substack{\{k, l\} \subset I_N \\ k \neq l}}  \mI_i(\Psi_{\{k,l\}}, P_{\{k,l\}}) \\
 = &   \E  \frac{1}{ 2 |B(0,N^{\frac13})|} \sum_{i \neq k \in I_N}  \mI_i(\Psi_{\{i,k\}}, P_{\{i,k\}}) 
 +  \E  \frac{1}{ 4 |B(0,N^{\frac13})|} \sum_{i \neq k \neq l \in I_N}   \mI_i(\Psi_{\{k,l\}}, P_{\{k,l\}}) \\
 = & \E  \frac{1}{ 2 |B(0,N^{\frac13})|} \sum_{i \neq k \in I_N}  \mI_i(\Psi_{\{i,k\}}, P_{\{i,k\}}) 
 \end{align*}
 using again \eqref{property_Ii}. 
This expectation can be reformulated thanks to the second order reduced moment measure  of the underlying point process $(x_i)_{i \in \N}$. We remind that for any $k \in \N^*$, the {\em $k-$th order reduced moment measure} is the symmetric measure $\rho_k$  over $(\R^{3})^k$ defined by: for all $F \in C_c(\R^{3k})$,  
 $$ \sum_{i_1 \neq i_2 \neq \dots \neq i_k} \E F(x_{i_1}, \dots, x_{i_k}) = \int_{(\R^3)^k} F(y_1, \dots, y_k) \rho_k(d y_1, \dots, d y_k). $$
In the case $\rho_k$ has more regularity, for instance when it has a density with respect to the Lebesgue measure, this relation can be extended to discontinuous $F$. 
In such a case, $\rho_k(d x_1, \dots, d x_k)  = \rho_k(x_1, \dots, x_k) dx_1 \dots dx_k$ can be roughly seen as the joint probability of having one particle in a volume $dx_1$ near $x_1$, one particle  in a volume $dx_2$ near $x_2$ and so on. Note  that here 
$$\rho_1(dx) = \rho_1(x) dx :=  \frac{3}{4\pi} \phi \, dx.$$ 
Also, under  \eqref{H1}, $\rho_k$ is supported outside the sets $\{ |x_i - x_j| \le R_0 \}$, $1 \le i \neq j \le k$. 

\mspace
For any $y \neq z$, we denote $\displaystyle B_y = \overline{B(y,1)}$,  and $\Phi_y$, resp. $\Phi_{y,z}$  the solution of 
\begin{equation} \label{corrector_y}
\begin{aligned}
 -\Delta \Phi_y  + \na P_y   & =  0,  \quad  x \in \R^3 \setminus B_y, \\
 \div \Phi_y & =  0, \quad  x \in \R^3 \setminus B_y, \\
D \Phi_y  + S & =  0, \quad  x \in B_y, \\
\int_{\pa B_y} \sigma(\Phi_y, P_y) n  & =  0, \quad \int_{\pa B_y} \sigma(\Phi_y, P_y) n \times (x-y)   =  0,
\end{aligned}
\end{equation}
resp. 
\begin{equation} \label{corrector_yz}
\begin{aligned}
 -\Delta \Phi_{y,z}  + \na P_{y,z}   & =  0,  \quad  x \in \R^3 \setminus \big(B_y\cup B_z\big), \\
 \div \Phi_{y,z} & =  0, \quad   x \in \R^3 \setminus \big(B_y\cup B_z\big),  \\
D \Phi_{y,z}   + S & =  0, \quad  x \in B_y \cup B_z, \\
\int_{\pa B_y} \sigma(\Phi_{y,z} , P_{y,z}) n  & = 0, \quad  \int_{\pa B_z} \sigma(\Phi_{y,z} , P_{y,z}) n =  0, \\ 
\int_{\pa B_y} \sigma(\Phi_{y,z}, P_{y,z}) n \times (x-y)  & = 0, \quad \int_{\pa B_z} \sigma(\Phi_{y,z}, P_{y,z}) n \times (x-z) = 0.
\end{aligned}
\end{equation}
Finally, we set 
\begin{equation} \label{def_Psi_I_xy}
 \Psi_{y,z} := \Phi_{y,z} - \Phi_y - \Phi_z, \quad \mI_x(u,p) := \int_{\pa B_x} \sigma(u,p)n \cdot Sn 
 \end{equation}
We find, still formally,
 \begin{align*}
 \mu_h S : S & = |S|^2 + \frac{5}{2}\phi |S|^2  +  \lim_{N \rightarrow +\infty} \frac{1}{ 2 |B(0,N^{\frac13})|} \int_{x \neq y \in B(0,N^{\frac13})}  \mI_x(\Psi_{x,y}, P_{x,y}) \rho_2(dx,dy)   + \dots 
\end{align*}
Our goal is to prove such a formula for a large class of stationary point processes. 
\subsection{Statement of the results}
Let us  formulate our assumptions on the point process $(x_i)_{i \in \N}$. First, in order to apply Propositions \ref{prop_corrector} and \ref{approx_effective_viscosity}, we assume as before that 
$(x_i)_{i \in \N}$ is an ergodic stationary point process satisfying  \eqref{H1}. In particular, stationarity implies that for all $k$ and all $x_1, \dots, x_k$ : 
$$\rho_k(x_1, \dots, x_k) = \rho_k(0,x_2 - x_1, \dots, x_k - x_1).$$

\mspace
Moreover, we make the following assumptions on the reduced moment measures of the process: there exists $q \in (1,\infty)$,  $F \in L^q(\R^3) \cap L^\infty(\R^3)$ going to zero at infinity, independent of $\phi$, such that for all $1 \le k \le 5$, $\rho_k = \phi^k g_k(x_1, \dots, x_k) dx_1 \dots d x_k$, with the  correlation functions $g_k$  satisfying

\mspace
\begin{align}  
\label{H2} \tag{H2}
g_2(0,y) & =  1 + R_1(y), \\
\nonumber
 |R_1(y)| \:  & \le \:  F\big(y\big)  \\
 \nonumber 
 & \\
 \label{H3} \tag{H3}
g_3(0,y,z) & =   g_2(0,y) \big(g_2(0,z) + R_2(y,z)\big) \\
\nonumber
|R_2(y,z)|  & \le \:  F\big(y-z\big), \\
\nonumber 
 & \\
\label{H4} \tag{H4}
g_4(0,y,z,z') & =  g_3(0,y,z)  \Big( \frac{g_3(0,y,z')}{g_2(0,y)}  +  R_3(y,z,z') \Big) \\
\nonumber
|R_3(y,z,z')| \: & \le \: F\big(z-z'\big), \\
 \nonumber 
 & \\
\tag{H5} \label{H5} 
g_5(0,y,y',z,z') & = \frac{g_4(0,y,y',z) g_4(0,y,y',z')}{g_3(0,y,y')} + R_4(y,y',z,z') \\ 
\nonumber
|R_4(y,y',z,z')| & \le F(z-z')
\end{align}
We shall comment on these assumptions just below. Note that in all quotients appearing in the previous expressions, the numerator is zero whenever the denominator is zero, and the convention is that  the quotient is zero in that case.  Note also that correlation functions $g_k$ may still depend on $\phi$, although it is not explicit in our notations. Our two results are the following: 
\begin{proposition} {\bf (Formula  for  the $\phi^2$ coefficient)} \label{main_prop}

\sspace
Under assumptions \eqref{H1}-\eqref{H2}, the formula 
\begin{equation} \label{def_mu2}
\mu_2 S : S = \lim_{N \rightarrow +\infty} \frac{1}{ 2 |B(0,N^{\frac13})|} \int_{B(0,N^{\frac13})^2}  \mI_x(\Psi_{x,y}, P_{x,y}) g_2(x,y) dx dy  \end{equation}
defines an element of $\text{Sym}\left(\text{Sym}_{3,\sigma}(\mathbb{R})\right)$ (in particular, the limit exists). Moreover, $\mu_2$ is bounded uniformly in 
$\phi$.   
\end{proposition}
\begin{theorem} \label{main_thm} {\bf (Derivation of the Batchelor-Green formula)}

\sspace
Under assumptions \eqref{H1} to \eqref{H5}
\begin{equation*}
 \mu_h   = \textrm{Id} + \phi \frac{5}{2} \textrm{Id}  + \phi^2 \mu_2 + O(\phi^{\frac52}) 
 \end{equation*}
with $\mu_2$ defined in \eqref{def_mu2}. 
\end{theorem}

\begin{remark}
Let us comment on Assumptions \eqref{H1} to \eqref{H5}. Assumption \eqref{H1}, slightly stronger than the non-penetration condition between the rigid spheres, relaxes a lot the  minimal distance assumption \eqref{Hstrong} under which previous studies were carried.

\mspace
As regards \eqref{H2} to \eqref{H5},  they contain two main pieces of information on the point process. The first one is on the amplitude of the correlation functions. As the function $F$ in our assumptions is bounded and independent of $\phi$, it follows that $g_k = O(1)$, or $\rho_k = O(\phi^k)$, consistently with the assumption made in \cite{BG}. Let us stress  again that the functions $g_k$ may depend on $\phi$, although such dependence is omitted in the notations. Our $L^\infty$ bound on $F$ also entails an $L^\infty$ control on quantities of the form $\frac{g_k(x_1,\dots,x_k)}{g_{k-1}(x_1,\dots,x_{k-1})}$, which is an extra constraint in the vicinity of points where $g_{k-1}$ vanishes. This assumption is made to avoid technicalities, but could be greatly relaxed. 

\mspace
The second piece of information is on the decay of correlations at infinity, through the decay of the function $F$. To understand it better, one can consider the important case of decorrelation at large distances, implying that there exists $R > 0$ such that for all $1 \le j \le k$ and all $x_1, \dots, x_k$, 
$$ g_k(x_1,\dots, x_k) = g_{j}(x_1,\dots,x_j) g_{k-j}(x_{j+1}, \dots, x_k) \quad \text{if }   \: \text{dist}\left(\{x_1, \dots, x_j \}, \{x_{j+1}, \dots, x_k\}\right) \ge R. $$
We claim that in such case, the identities in  \eqref{H2}-\eqref{H5} are satisfied with remainders $R_1$, resp. $R_2$, resp. $R_3, R_4$,  that vanish for $y$, resp. $y-z$, resp. $z-z'$ large enough. For instance, the identity in \eqref{H2} clearly holds with $R_1(y) = 0$ for $|y| \ge R$. Similarly, the identity in \eqref{H3} holds with $R_3(y,z) = 0$ for $|y-z| \ge 2R$. Indeed, it is enough to distinguish between two cases: 
\begin{itemize}
\item either $|z| \le R$, then $\text{dist}\left(y, \{0,z\}\right) \ge R$, which implies 
$$g_3(0,y,z) = g_2(0,z) g_1(y) =  g_2(0,z) g_2(0,y)$$
\item or $|z| \ge R$, then $\text{dist}\left(z, \{0,y\}\right) \ge R$, which implies again
$$g_3(0,y,z) = g_2(0,y) g_1(z) = g_2(0,y) g_2(0,z).$$  
\end{itemize}
Similar but lengthier reasoning allows to treat \eqref{H4} and \eqref{H5} as well. 

\mspace
Hence, our condition is in the spirit of other quantitative notions of decay of correlations, such as $\alpha$-mixing or $\Phi$-mixing, {\it cf.} \cite{Bradley}. Still,  direct comparison to these more classical notions is uneasy. Let us insist though that our hypotheses on the correlation functions $g_k$ are only required for $0 \le k \le 5$: one may consider arbitrary $6$-point correlations and more,  while $\alpha$-mixing or $\Phi$-mixing conditions apply to all correlations.

\mspace
Eventually, let us point out that  our hypotheses apply to classical hardcore Poisson processes, sometimes called Mat\'ern processes \cite{Bart}. They are obtained from usual Poisson point processes by various thinning procedures, in order to satisfy the hardcore condition \eqref{H1}. The simplest example is the  so-called Mat\'ern process of type 1 :  one deletes from the Poisson process all points $x_i$  that are not isolated in $B(x_i, R_0)$. The resulting process obeys all our assumptions. Indeed,  it satisfies \eqref{H1} by construction. As it has finite range of correlations,$R_1$ is zero for $|y|$ large, $R_2$ zero for $|y-z|$ large, {\it etc}.  It then remains to show that the $R_i$'s are bounded in $L^\infty$ uniformly in $\phi$, which follows from the boundedness of the $g_k$'s and the boundedness of the ratios 
$\frac{g_k(x_1,\dots,x_{k-1,x_k})}{g_{k-1}(x_1,\dots,x_{k-1})}$. The first property comes from the fact that $\rho_k \le \phi_P^k$, where $\phi_P$ is the intensity of the initial Poisson process, which scales like $C \phi$ for small $\phi$, see \cite{Bart}. The second one comes from the fact that $\rho_k(x_1, \dots, x_k) \le \rho_{k-1}(x_1, \dots, x_{k-1}) \phi_P$.
\end{remark}

\begin{remark}
Short before completion of this work, we got aware of the very nice preprint by Duerinckx and Gloria \cite{DueGloria}, in which they perform an extensive study of the effective viscosity of random suspensions. This study contains two parts. The first one is dedicated to two somehow extreme examples of point processes : those obtained by random deletion (so that $\rho_k = \phi^k g_k$ for all $k \in \N$, with $g_k$ independent of $\phi$), and those satisfying a strong separation assumption of type \eqref{Hstrong}. It notably shows analyticity in $\phi$ of the effective viscosity for the first example. The second part of the paper culminates in  Theorem 7, showing the  existence of an expansion in $\phi$ of the effective viscosity tensor at arbitrary order,  for a large class of point processes. It validates as a special case both the Einstein's formula and Batchelor and Green's correction. The key assumption there is that the approximations  $\mu_{h,L}$ of $\mu_h$  obtained by periodization over a grid of size $L$ converges to $\mu_h$ with explicit rate $L^{-\gamma}$ for some $\gamma > 0$.  It is shown that such assumption is satisfied under an $\alpha$-mixing condition on the point process, with $\alpha(t) = t^{-a}$, $a > 0$. 

\mspace
Although the analysis in \cite{DueGloria} goes much further than the one presented here, we feel that the latter may be of independent interest, both by its results and by its method. First, our setting is not covered by  the random deletion case, as $g_k$ is allowed to depend on $\phi$ (and is even fully arbitrary for  $k \ge 6$). It is nor covered by  \cite[Theorem 7]{DueGloria}:  the kind of mixing conditions that we use, only put on the first correlation functions of the process,  are not the same as the $\alpha$-mixing used there, see the discussion in the previous remark.  Furthermore, the general bounds given in \cite[Theorem 7]{DueGloria} would only imply in our context $\mu_2 = O(|\log \phi|)$, to be compared to the optimal bound $\mu_2 = O(1)$ in Proposition \ref{main_prop}. The reason behind the logarithm loss in \cite{DueGlo},  is that the  bounds obtained there on the periodic approximations of size $L$ have a slight divergence with $L$. The assumption of quantitative convergence of $\mu_{h, L}$ to $\mu_h$ is then used to circumvent this problem, by appropriately linking $L$ to $\phi$, or in our context linking $N$ to $\phi$, as there is  a parallel between $L$ and $N^{\frac13}$. On the contrary, the bounds shown in the present paper   are  uniform in $N$. To derive these bounds, we must control uniformly in $N$ various integrals taken over balls of radius $N^{1/3}$, similar to the one at the right-hand side of \eqref{def_mu2}. The problem is that the corresponding  integrands are only borderline integrable on the whole space. Still, identifying Calderon-Zygmund kernels within these integrands, we are able to show uniform boundedness using several times the $L^p$ continuity of Calderon-Zygmund operators. The derivation of an estimate with a slight loss in $N$ (namely a power of $(\ln N)$) would not require so much care and would be much simpler.

\mspace
Eventually, the proof presented here turns out to be very different and much more elementary than the one in \cite{DueGloria}, which relies on deep results in quantitative homogenization obtained over the last years (see for instance the Appendix A, showing that $\alpha$-mixing implies algebraic convergence of periodic approximations of the effective viscosity tensor). It has to be said however that  implementation of our proof to arbitrary high order approximations may be hard. For instance, to understand the right analogue of Propositions  \ref{prop_large} and Proposition \ref{prop_short} for more than two balls may be difficult. 
\end{remark}

\section{Proofs}
This long section is devoted to the proof of our results.  The main difficulty is in establishing Theorem \ref{main_thm}, that is a $o(\phi^2)$ approximation of the effective viscosity. It will be deduced from a $o(\phi^2)$ approximation of $u_{I_N,S}$ solution of \eqref{stokes} with $I = I_N = \{ i, B_i \subset B(0,N^{\frac13})\}$.  This approximation $u_{N,app}$ will be built using the cluster expansion introduced at a formal level in the previous section. Then, roughly, we will show that  
$$ \frac{1}{N}\E\|D(u_{I_N,S} -u_{N,app})\|^2_{L^2(\R^3)} \lesssim \frac{1}{N} \E\|D(u_{I_N,S} -u_{N,app})\|^2_{L^2(\cup B_i)} $$
and  prove that the latter quantity is $o(\phi^2)$. As   $u_{I_N,S}$ and $u_{N,app}$ are fully explicit at the boundary, this latter term is fully explicit : it is composed of integrals similar to \eqref{def_mu2} involving solutions of one-sphere and two-spheres Stokes problems, as well as the correlation functions $g_k$ of the point process, $2 \le k \le 5$. The whole point is to show that these integrals (taken over products of the ball $B(0, N^\frac13)$ are uniformly bounded in $N$. One main difficulty is that the integrands are only borderline integrable over the whole space, so that possible logarithmic divergence in $N$ may occur. To obtain a uniform control, one must study very carefully the structure of the integrands. This is the purpose of the preliminary paragraphs \ref{prelim1} and \ref{prelim2}. Broadly, the goal is  to decompose the correlation functions and the solution of the two-sphere problems into more elementary tensorized quantities,  in which  the variables are partly separated. For instance,  as we will show in paragraph \ref{prelim2}, the solution of a two-sphere Stokes problem may be expressed in terms of solutions of one-sphere Stokes problems in some asymptotic regimes. The general idea is that,  using a decomposition in terms of these solutions, which involve naturally Calderon-Zygmund kernels, we will  overcome the problem of borderline integrability, for instance  through the use of a Calderon-Zygmund type theorem. The decay of correlations given in assumptions \eqref{H2}-\eqref{H5} will also be there to provide a little gain for integrability.

\subsection{Reformulation of the assumptions} \label{prelim1}
For some  of our arguments,  we will need to go beyond the direct formulation  \eqref{H3}-\eqref{H5}, and use the following refined formulas: 
\begin{proposition}
Under the assumptions \eqref{H2}-\eqref{H5}, one can write 
\begin{align}  
\label{H3'} \tag{H3'}
g_3(0,y,z) & =   g_2(0,y) g_2(0,z)  + (g_2(y,z) - 1) g_2(0,z)  + \tilde{R}_2(y,z), \\
\nonumber
|\tilde{R}_2(y,z)| \:  & \le \:  F\big(y-z\big) F\big(y\big) \\
\nonumber 
 & \\
\label{H4'} \tag{H4'}
g_4(0,y,z,z') & =   \frac{g_3(0,y,z) g_3(0,y,z')}{g_2(0,y)} +  \frac{g_3(0,z,z') g_3(y,z,z')}{g_2(z,z')} \\
\nonumber
& - g_2(0,z) g_2(y,z) g_2(0,z') g_2(y,z')  +  \tilde{R}_3(y,z,z'),  \\
\nonumber
|\tilde{R}_3(y,z,z')|  \: & \le \:  F(y)  F\big(z-z'\big), \\
 \nonumber 
 & \\
\tag{H5'} \label{H5'} 
g_5(0,y,y',z,z') &  = \frac{g_4(0,y,y',z) g_4(0,y,y',z')}{g_3(0,y,y')} +  \frac{g_4(0,y,z,z') g_4(0,y',z,z')}{g_3(0,z,z')}  \\
\nonumber
& - \frac{g_3(0,y,z)  g_3(0,y',z)  g_3(0,y,z')  g_3(0,y',z')}{g_2(0,y) g_2(0,y') g_2(0,z) g_2(0,z')} + \tilde{R}_4(y,y',z,z') \\ 
\nonumber
|{\tilde R}_4(y,y',z,z')| & \le F(y-y')  F(z-z')
\end{align}
\end{proposition}
\begin{proof}
For brevity, we only prove \eqref{H5'}, which is the most elaborate.  Let ${\tilde R}_4$ as defined by the equality in \eqref{H5'}.  We actually only need to prove the following two inequalities:  
\begin{align*}
|{\tilde R}_4(y,y',z,z')| \le F(z-z')  \quad \text{ uniformly in $y,y'$}, \quad \text{for some $F \in L^q \cap L^\infty$}, \\
|{\tilde R}_4(y,y',z,z')| \le F(y-y')  \quad \text{ uniformly in $z,z'$}, \quad \text{for some $F \in L^q \cap L^\infty$}, 
\end{align*}
Indeed, it follows that $|{\tilde R}_4(y,y',z,z')|  \le \sqrt{F(y-y')}  \sqrt{F(z-z')}$, so that we get the result with $F := \sqrt{F}$, $q := 2 q$.   By symmetry, it is enough to show the first bound. By \eqref{H5}, we  have 
$$  |g_5(0,y,y',z,z') - \frac{g_4(0,y,y',z) g_4(0,y,y',z')}{g_3(0,y,y')}| \le  F(z-z') $$
It remains to handle the second part in ${\tilde R}_4$. By \eqref{H2}-\eqref{H4}, it is clearly bounded, so that the difficulty is only for  $z-z'$ large. In particular, either $z$ is large or $z'$ is large. By symmetry, we can always assume that $z'$ is large, and so by \eqref{H2} that $\rho_2(0,z') \ge \frac{1}{2}$. Pondering on \eqref{H2}-\eqref{H4}, we find 
\begin{align}
& \frac{g_4(0,y,z,z') g_4(0,y',z,z')}{g_3(0,z,z')} - \frac{g_3(0,y,z)  g_3(0,y',z)  g_3(0,y,z')  g_3(0,y',z')}{g_2(0,y) g_2(0,y') g_2(0,z) g_2(0,z')} \\
= \: & \frac{ g_3(0,y,z) \big(  \frac{g_3(0,y,z')}{g_2(0,y)} + R_3(y,z,z') \big) g_3(0,y',z)  \big(\frac{g_3(0,y',z')}{g_2(0,y')} + R_3(y',z,z') \big)}{g_2(0,z) (g_2(0,z') + R_2(z,z'))} \\
 -  \: &  \frac{g_3(0,y,z)  g_3(0,y',z)  g_3(0,y,z')  g_3(0,y',z')}{g_2(0,y) g_2(0,y') g_2(0,z) g_2(0,z')} 
\end{align}
Note that we can always assume $z-z'$ large enough and  $z'$ large enough so that 
$$\frac{R_2(z,z')}{g_2(0,z')} \le \frac{1}{2},  \quad \text{resulting in } \quad \frac{1}{g_2(0,z') + R_2(z,z')} = \frac{1}{g_2(0,z')} (1 + R(z,z')) $$
 with $|R(z,z')| \le F(z-z')$, $F \in L^q \cap L^\infty$. Expanding all terms, we see that the $O(1)$ contributions cancel, and we end up with 
$$ \Big| \frac{g_4(0,y,z,z') g_4(0,y',z,z')}{g_3(0,z,z')} - \frac{g_3(0,y,z)  g_3(0,y',z)  g_3(0,y,z')  g_3(0,y',z')}{g_2(0,y) g_2(0,y') g_2(0,z) g_2(0,z')} \Big| \le F(z-z') $$
for some $F \in L^q \cap L^\infty$, for some finite $q$. Hence, the same bound holds for ${\tilde R}_4$, which concludes the proof. 
\end{proof}

\subsection{Preliminaries about the two-sphere configuration} \label{prelim2}
As mentioned earlier, 
$$ \Phi_y(x) = \Phi_0(x-y) $$
with $\Phi_0$ defined in  \eqref{def_Phi0}. Note that $\Phi_0$ is linear in $S$, with main part  homogeneous of order $-2$. It implies that 
\begin{equation} \label{def_mM0}
D \Phi_0(x-y) = \mM_0(x-y) S
\end{equation}
where $\mM_0 = \mM_0(x) \in  \text{Sym}\left(\text{Sym}_{3,\sigma}(\mathbb{R})\right)$, with  main part homogeneous of order $-3$, of Calderon-Zygmund type (see the discussion in  appendix \ref{appendix_CZ}). 

\mspace
We now establish several properties of the field $\Psi_{y,z}$ introduced in  \eqref{def_Psi_I_xy}. Clearly, the following symmetry properties hold: 
\begin{lemma}  {\bf (Symmetry properties)} \label{lem_sym}
\begin{align*}
\Psi_{y,z}(x) =  \Psi_{z,y}(x) =  \Psi_{z-x,y-x}(0), \quad \Psi_{y,z}(-x) = - \Psi_{-y,-z}(x) 
\end{align*}

\end{lemma}

\mspace
 The next two propositions specify the behaviour of $\Psi_{y,z}$ when the sphere centers $y$ and $z$ are far from one another  and when they are close from one another compared to their distance to $x$.

\begin{proposition}  {\bf (Behaviour for large distance of the spheres centers)} \label{prop_large}

\mspace
There exists $R_1 > R_0$ such that for all $x,y,z$ with  $|y-z| \ge R_1$, $\: |x - y|, |x - z| \ge\frac{2+R_0}{4}$, 
\begin{align*}
 D \Psi_{y,z}(x) = \Big( \mM_0(x-y) + \mM_0(x-z) \Big)  \mM_l(y-z) S    +  R_{y,z}(x) 
\end{align*}
where $\mM_l$ and $R_{y,z}(x)$ are smooth in their arguments, with: 
\begin{align*}
|\mM_l(y-z) - \mM_0(y-z)|  & \le  C |y-z|^{-4},  \\
|R_{y,z}(x)| & \le  C \big(|x-y|^{-4} |y-z|^{-4} + |x-z|^{-4} |y-z|^{-4}\big) 
\end{align*}
\end{proposition}

\begin{proof}
This proposition relies on the fact that for $|y-z|$ large enough, the solution $\Phi_{y,z}$ of \eqref{corrector_yz} can be solved thanks to the method of reflections. This method, that goes back to Smoluchowski, was studied in great details in article \cite{Hoferter}. The analysis in \cite{Hoferter} is about systems that differ slightly  from  \eqref{corrector_yz}, resp.   \eqref{stokes} :  the inhomogeneous boundary data on the balls, resp. inhomogeneous boundary condition at infinity, is replaced by an inhomogeneous source term in the Stokes equation. Still, the iteration process described there extends straightforwardly to \eqref{corrector_yz} or  \eqref{stokes}, only the initialization being different. In particular,  as a consequence of this analysis, one can express $\Phi_{y,z}$ as a series. 
\begin{equation} \label{series}
\begin{aligned} 
\Phi_{y,z}  & =  \Big(Q_y[S\cdot] + Q_z[S\cdot]\Big) \:  +  \: \Big(Q_y Q_z[S\cdot] + Q_z Q_y[S\cdot])  \: + \:     \Big(Q_y Q_z Q_y[S\cdot] + Q_z Q_y Q_z[S\cdot]\Big)  \\
& +  \Big(\underbrace{Q_y Q_z \dots}_{\text{$k$ factors}}[S\cdot] +\underbrace{Q_z Q_y \dots}_{\text{$k$ factors}}[S\cdot] \Big) \: + \:  \dots 
\end{aligned}
\end{equation}
where the operator $Q_y$ is defined by : 
for all $y \in \R^3$,  for all $w \in H^1_{\sigma,loc}(\R^3)$, $u  = Q_y[w]$ is the solution in $\dot{H}^1_\sigma(\R^3)$ of 
\begin{equation*} 
\begin{aligned}
-\Delta u  + \na p  & = 0,  \quad  x \in \R^3 \setminus B_y, \\
\div u   & = 0,  \quad  x \in \R^3  \setminus B_y, \\
D(u + w) & = 0,  \quad  x \in B_y, \\
\int_{\pa B_y} \sigma(u, p) n   &  =  \int_{\pa B_y} \sigma(u, p) n \times (x-y) & = 0. \\
\end{aligned}
\end{equation*}
From \cite{Hoferter}, one can infer that the series in \eqref{series} converges in $\dot{H}^1_\sigma(\R^3)$ for $|y-z| \ge R_1$, $R_1$ large enough. Actually, the convergence result in \cite{Hoferter} is much more general, as an arbitrary number of balls is considered. The case of two balls considered here could be handled more directly, but we shall not expand on that. Let us stress that in the region $\{x,\:  |x-y|,  |x-z| \ge  \frac{2+R_0}{4}\}$, the partial sums of the series are solutions of the homogeneous Stokes equation. By classical elliptic regularity estimates, one can deduce that  the convergence holds uniformly in this region for all derivatives   of order $\ge 1$. 

\mspace
Noticing that  $Q_{y}[S\cdot] = \Phi_y$, we end up with 
\begin{align*} 
\Psi_{y,z}(x)   = & + Q_y Q_z [S\cdot + Q_y[S\cdot] + Q_y Q_z[S\cdot] + \dots]   \\
& + Q_z Q_y [S\cdot + Q_z[S\cdot] + Q_z Q_y[S\cdot] + \dots] =  \Psi^1_{y,z}(x) + \Psi^2_{y,z}(x) 
\end{align*}
Clearly, to prove the proposition, it is enough to show that  
$$  D \Psi^1_{y,z}(x)   = \mM_0(x-y)  \mM_l^1(y-z) S  +  R^1_{y,z}(x) $$
where $\mM_l^1$ and $R^1_{y,z}(x)$ have the same behaviour as $\mM_l$ and $R_{y,z}(x)$  in the proposition. Indeed, as $\Psi^2_{y,z} = \Psi^1_{z,y}$, we get 
$$ D \Psi_{y,z}(x)   = \mM_0(x-y)  \mM_l^1(y-z) S   \mM_0(x-z)  \mM_l^1(z-y) S  +   R^1_{y,z}(x)  + R^1_{z,y}(x) $$
Moreover, using that $ D \Psi_{y,z}(x) = D(\Psi_{-y,-z})(-x)$, {\it cf.} Lemma \ref{lem_sym} and the parity of $\mM_0$, we obtain the appropriate formula, with 
\begin{align*}
\mM_l := \frac{1}{2} \left(  \mM^1_l +   \mM^1_l(-\cdot) \right), \quad  R_{y,z}(x) = \frac{1}{2}  \left(  R^1_{y,z}(x) + R^1_{z,y}(x) + R^1_{-y,-z}(-x) +  R^1_{-z,-y}(-x) \right) 
\end{align*}

\mspace
To do so, we shall rely on the following properties:
\begin{lemma} \label{properties_Q} {\bf (\cite[Lemmas 4.3 and 4.4]{Hoferter})}
For all $y \in \R^3$, for all $|x-y| > 1$, for all $S \in \text{Sym}_{3,\sigma}(\mathbb{R})$, 
$$D Q_y[S\cdot](x) = \mM_0(x-y)S.$$
For all $y \in \R^3$, for all $|x-y| > 1$, for all $w \in \dot{H}^1_{\sigma,loc}(\R^3)$ with  $\dashint_{B_y} D w = 0$,
$$|D Q_y[w](x)| \le \frac{C}{|x-y|^4} ||D w||_{L^2(B_y)}.$$ 
Hence, for any $w$, 
\begin{equation} \label{ineq_lemma_properties_Q}
D Q_y[w](x) = \mM_0(x-y) \Big( \dashint_{B_y} D w \Big)  + O\Big(|x-y|^{-4} ||D w -  \dashint_{B_y} D w||_{L^2(B_y)}\Big).
\end{equation}
\end{lemma}

\mspace
 Let 
$$w :=  Q_y[S\cdot] + Q_y Q_z[S\cdot] + \dots$$
 so that $\Psi^1_{y,z} =  Q_y Q_z[S\cdot]   Q_y Q_z[w]$. Clearly, by the first property in Lemma  \ref{properties_Q}, 
 \begin{align*}
\forall \tilde x \in B_y, \quad D Q_z[S \cdot](\tilde x) & =   \mM_0(\tilde x-z) S  \\
&  = \dashint_{B_y}  \mM_0(x'-z) S dx'   \: + \:  \Big(  \mM_0(\tilde x-z) S - \dashint_{B_y}  \mM_0(x'-z) S dx'  \Big)  
\end{align*}
where the last term has zero mean and is bounded by $C/|y-z|^4$. By Lemma \ref{properties_Q}, 
\begin{align*} 
D Q_y Q_z[S\cdot](x) &  = \mM_0(x-y) \dashint_{B_y}  \mM_0(x'-z) S dx'  +  O(|x-y|^{-4} |y-z|^{-4} ) \\
& = \mM_0(x-y) \mM_l^a(y-z) +  O(|x-y|^{-4} |y-z|^{-4} ) 
\end{align*}
where $\mM_l^a := \dashint_{B_y}  \mM_0(x'-z) S dx'$ satisfies   
$$\mM_l^a(y-z) = \mM_0(y-z) S + O(|y-z|^{-4}).$$ 
For the other part of $\Psi^1_{y,z}$, we remark that by translational invariance, $Q_z[w](x')$ can be seen as a function of  $x'-y$ and $y-z$ only: $Q_z[w](x') = W(x'-y, y-z)$. It follows that its average on the ball $B_y$ can be seen as a function of $y-z$. Using the linearity in $S$, we end up with an expression of the form 
$$\dashint_{B_y}  D Q_z[w](x') dx' =: \mM_l^b(y-z)S. $$
Also, as $w$ is of the form $Q_y[\tilde w]$, it satisfies for all $\tilde x \in B_z$, 
$$ |D w(\tilde x)| \le C/|y -z|^3. $$
we deduce that for all  $\tilde x \in B_y$, 
$$ |D Q_z[w] (\tilde x)| \le   C/|y -z|^6,  $$
and in particular, 
$$ ||D Q_z[w]  -  \dashint_{B_y} D Q_z[w] ||_{L^2(B_y)}\Big)\le  C/|y -z|^6.  $$
Thanks to this last inequality and inequality  \eqref{ineq_lemma_properties_Q} in Lemma  \ref{properties_Q} (replacing $w$ by $Q_z w$), we end up with 
$$ D Q_y Q_z[w](x) =  \mM_0(x-y)  \mM_l^b(y-z)S + O(|x-y|^{-4} |y -z|^{-6}). $$
Setting $\mM^1_l := \mM_l^a + \mM_l^b$ concludes the proof. 
\end{proof}

\begin{proposition} {\bf (behaviour for short distance of the spheres centers)} \label{prop_short}

\mspace
There exists $R_2 > 2R_1 + R_0$ such that for all $x,y,z$ with  $\frac{2+R_0}{4} \le |y-z| \le 2 R_1$, $|y - x| \ge R_2$, 
\begin{align*}
 D \Psi_{y,z}(x) =  \Big(\mM_0(x-y) + \mM_0(x-z)\Big)  \mathcal{M}_s(y-z) S +  R_{y,z}(x) 
\end{align*}
where $\mM_s$ and $R_{y,z}(x)$ are smooth in their arguments, with 
\begin{align*} 
R_{y,z}(x) & = O(|x-y|^{-4}), \quad   |x-y|   \rightarrow +\infty 
\end{align*}
\end{proposition}
\begin{remark}
With our choice of $R_2$, if $|x-y| \ge R_2$ and $|y-z| \le 2 R_1$, then $|x-z| \ge R_0$. 
\end{remark}

\begin{proof}
The proof is an adaptation of arguments given by M\'echerbet in the context of sedimentation problems \cite{MR4127954}. We give the main ideas and refer to \cite{MR4127954} for complementary details. First, by translation, 
rotation and rescaling, we can always assume that $y=0$, $z=2$, and the point is to prove that 
$$ D \Psi_{0,2}(x) = (\mM_0(x) + \mM_0(x-2)) \mN S + O(|x|^{-4}) $$
for some tensor $\mN$.
We recall that $D\Phi_0(x) = \mM_0(x)S$,  $D\Phi_2(x) = \mM_0(x-2)S$. Moreover, clearly, $\mM_0(x-2) = \mM_0(x) + O(|x|^{-4})$.  Hence, it is enough to show that 
$$D \Phi_{0,2}(x) = \mM_0(x)\mN  + O(|x|^{-4}) $$
(see \eqref{corrector_yz} for the definition of $\Phi_{y,z}$).  We denote $u = \Phi_{0,2}$ for brevity, and introduce a smooth $\chi$ such that $\chi = 0$ in $B(0,4)$ (vicinity of the two spheres), and $\chi = 1$ outside $B(0,6)$. As in \cite{MR4127954}, we set 
$$ \overline{u} = \chi u + \tilde u, \quad \overline{p} = \chi p  $$
where $\tilde u$ is a corrector added to recover the divergence-free condition: $\div \tilde u = - u \cdot \na \chi$. By classical considerations involving the Bogovski operator, as $u \cdot \na \chi$ is smooth, compactly supported in $B(0,6) \setminus B(0,4)$,  and linear in $S$,  so can be taken $\tilde u$. 
Also, as $\overline{u}$ is smooth and divergence free on $\R^3$, it can always be written as a solution of an inhomogeneous Stokes equation in $\R^3$: 
$$ -\Delta \overline{u} + \na \overline{p} = f, \quad \div \overline{u} = 0.$$  
Note that $f$ is compactly supported in $B(0,6)$. Introducing the Oseen tensor  
$$O_{s}(x) := \frac{1}{8\pi} \left( \frac{\textrm{Id}}{|x|} + \frac{x \otimes x}{|x^3} \right),$$  
\begin{align*}
 \forall |x| > 6, \quad  u(x) & = \overline{u}(x) =   \int_{B(0,6)} O_s(x-y) f(y) dy \\
 & = O_{s}(x) \int_{B(0,6)} f(y) dy  - \sum_{l=1}^3 \pa_l O_s(x) \int_{B(0,6)} y_l f(y) dy + O(|x|^{-3}) 
 \end{align*}
using a Taylor expansion of the Oseen tensor at $x$. We claim that 
$$ \int_{B(0,6)} f(y) dy = 0 $$
and that 
$$ S_f :=  \int_{B(0,6)} f(y) \otimes y \, dy := \Big(   \int_{B(0,6)}  y_j  f_i(y) dy \big)_{i,j} \text{ is symmetric and trace-free.} $$  
The first point was established in \cite{MR4127954}. Denoting $\Sigma := 2 D u - p \textrm{Id}$, $ \overline{\Sigma}  := 2 D \overline{u} - \overline{p} \textrm{Id}$, we find 
\begin{align*}
\int_{B(0,6)} f(y) dy & =  \int_{B(0,6)} \div \overline{\Sigma}  = \int_{\pa B(0,6)}  \overline{\Sigma}n  = \int_{\pa B(0,6)}  \Sigma n \\
  & =  \int_{B(0,6)\setminus (B_0 \cup B_2)}  \div \Sigma \: + \:  \int_{\pa B_0 \cup \pa B_2}  \Sigma n  = 0, 
\end{align*}
where the last equality comes from the first and  fourth  line of \eqref{corrector_yz}. For the second point,  we compute 
\begin{align*}
S_f  & =  \int_{\pa B(0,6)} \overline{\Sigma}n \otimes y  - \int_{B(0,6)} \overline{\Sigma}    =\int_{\pa B(0,6)} \Sigma n \otimes y  - \int_{B(0,6)} \overline{\Sigma}  \\
  & = \int_{B(0,6)\setminus (B_0 \cup B_2)}  (\div \Sigma) \otimes y + \int_{\pa B_0 \cup \pa B_2} \Sigma n \otimes y  + \int_{B(0,6)} \Sigma -   \int_{B(0,6)} \overline{\Sigma} \\ 
  & =    \int_{\pa B_0 \cup \pa B_2} \Sigma n \otimes y   +   \int_{B(0,6)} \Sigma  \: - \:    \int_{B(0,6)} \overline{\Sigma} 
\end{align*}  
where the last equality comes from the first line of \eqref{corrector_yz}.  The last two terms at the r.h.s. are symmetric, as $\Sigma$ and $\overline{\Sigma}$ are. Moreover, the relation $\int_{\pa B_0 \cup \pa B_2} \Sigma n \times y = 0$, deduced from the fourth and fifth lines of \eqref{corrector_yz}, implies that the first term is symmetric. Hence, $S_f$ is symmetric. Moreover, it can always be assumed to be trace-free by a proper normalization of the pressure. 

\mspace
Back to the asymptotic expansion of $u(x)$, taking into account the fact that $S_f$ is symmetric and trace-free, a little calculation shows the formula:  
\begin{align*} 
u(x) & = - \sum_{l=1}^3 \pa_l O_s(x) \int_{B(0,6)} y_l f(y) dy + O(|x|^{-3})  = \frac{3}{8\pi} \frac{S_f : (x \otimes x)}{|x|^5} + O(|x|^{-3}).    
\end{align*}
Note that this term is close to the expression of $\Phi_0$ in  \eqref{def_Phi0}, replacing $S$ by $S_f$ and neglecting the $O(|x|^{-5})$ terms. More precisely, considering the symmetric gradients, we get 
$$ D u(x) = - \frac{3}{20\pi} \mM_0(x) S_f + O(|x|^{-4}) $$
which concludes the proof. 
\end{proof}

\bspace
To shorten the analysis of integrals involving $ D \Psi_{y,z}(x)$, it will be convenient to sum up the two previous propositions into a single one: 
%
%

\begin{corollary} \label{cor_global_psi_xyz} {\bf (Global behaviour)}

\sspace
Let $\chi_0$ a smooth radial function with $\chi_0 = 0$ in $B(0,\frac{2+R_0}{4})$, $\chi_0=1$ outside $B(0,R_0)$.  The function 
\begin{equation} \label{def_S_xyz}
S_{y,z}(x) :=  \chi_0(x-y) \chi_0(x-z) \chi_0(y-z)    D \Psi_{y,z}(x) 
\end{equation}
satisfies 
\begin{align*}
   S_{y,z}(x) 
=  \Big([\chi_0 \mM_0](x-y) \chi_0(x-z) + [\chi_0 \mM_0](x-z) \chi_0(x-y) \Big)  \mathcal{N}_1(y-z) S +  R_{y,z}(x) 
\end{align*}
where $\mN_1, R_{y,z}(x)$ are smooth in their arguments, with   ($\langle \xi \rangle := \sqrt{1+|\xi|^2}$): 
\begin{align*}
|\mN_1(x-y) - \chi_0 \mM_0(x-y)| & \le C \langle x-y \rangle^{-4},  \\
|R_{y,z}(x)| & \le C \big(\langle x-y\rangle^{-4} \langle y-z\rangle^{-4} + \langle x-z\rangle^{-4} \langle y-z\rangle^{-4}\big),
\end{align*}
\end{corollary}
\begin{remark}
Corollary \ref{cor_global_psi_xyz} implies the crude bound 
\begin{equation} \label{crude_bound}
|S_{y,z}(x)| \le C   \Big( \frac{1}{\langle x-y \rangle^3} +   \frac{1}{\langle x-z\rangle^3} \Big)    \frac{1}{\langle y-z\rangle^{3}}, 
\end{equation}
that we will  use several times in the rest of the paper. 
\end{remark}

\begin{proof}
Let $R > 0$, and $\chi_R$  a smooth radial function with $\chi_R = 0$ in $B(0,R)$, $\chi_R=1$ outside $B(0,2R)$. We decompose
\begin{align*}
  S_{y,z}(x) 
= & \chi_0(x-y) \chi_0(x-z) \chi_0(y-z)    D \Psi_{y,z}(x)  \\
= &  \chi_0(x-y) \chi_0(x-z)   \chi_{R_1}(y-z)   D \Psi_{y,z}(x) \\
+ &  \chi_{R_2}(x-y)   \chi_0(x-z) (\chi_0-\chi_{R_1})(y-z)   D \Psi_{y,z}(x) \\
+ &  (\chi_0 - \chi_{R_2})(x-y) \chi_0(x-z)   (\chi_0-\chi_{R_1})(y-z)    D \Psi_{y,z}(x) \\
= & I_{y,z}(x) + J_{y,z}(x) + K_{y,z}(x). 
\end{align*}
The term $K_{y,z}(x)$ can be put in the remainder. For  $I_{y,z}(x)$, we use Proposition \ref{prop_large}: 
\begin{align*}
I_{y,z}(x) =&   \chi_0(x-y) \mM_0(x-y)  \chi_0(x-z)  \chi_{R_1}(y-z)  \mM_l(y-z) \\
& + \chi_0(x-z) \mM_0(x-z)  \chi_0(x-y)  \chi_{R_1}(y-z)  \mM_l(y-z) + \text{remainder} 
\end{align*}
Finally, for  $J_{y,z}(x)$, we use Proposition \ref{prop_short}: 
\begin{align*}
J_{y,z}(x) = &    \chi_{R_2}(x-y)   \mM_0(x-y)   \chi_0(x-z)    (\chi_0-\chi_{R_1})(y-z)  \mM_s(y-z) \\
& +  \chi_0(x-z)  \mM_0(x-z)     \chi_{R_2}(x-y)   (\chi_0-\chi_{R_1})(y-z)   \mM_s(y-z) + \text{remainder} \\
 = &    \chi_{0}(x-y)  \mM_0(x-y)   \chi_0(x-z)    (\chi_0-\chi_{R_1})(y-z)  \mM_s(y-z) \\
& +  \chi_0(x-z)  \mM_0(x-z)     \chi_0(x-y)  (\chi_0-\chi_{R_1})(y-z)   \mM_s(y-z) + \text{remainder}   \\
\end{align*}
The corollary follows, by setting: 
\begin{align*}
 \mN_1 =  \chi_{R_1}  \mM_l +   (\chi_0-\chi_{R_1}) \mM_s 
\end{align*}
\end{proof}

\begin{corollary} \label{cor_rho_2}
For all $ 1 \le q < \infty$, for all $G_1, G_2, G_3 \in   \R +  (L^q(\R^3) \cap L^\infty(\R^3))$, and for all  $m \in (1,\infty)$, there exists $C > 0$ such that 
\begin{equation*}
\int_{B(0,N^{\frac13})} \int_{B_x}   \Big| \int_{B(0,N^{\frac13})^2}  S_{y,z}(x') G_1(x-y) G_2(x-z) G_3(y-z)  dy dz \Big|^m dx'      dx \le  C  N 
\end{equation*}
\end{corollary}

\begin{proof}
We use Corollary \ref{cor_global_psi_xyz}: 
\begin{multline*}
 \int_{B(0,N^{\frac13})} \int_{B_x}   \Big| \int_{B(0,N^{\frac13})^2} S_{y,z}(x') G_1(x-y) G_2(x-z) G_3(y-z)  dy dz \Big|^m dx'      dx \le  C (I_1 + I_2 + I_3) 
\end{multline*}
where
\begin{align*}
I_1 & := \int_{B(0,N^{\frac13})}  \int_{B_x}   \Big| \int_{B(0,N^{\frac13})^2} [\chi_0\mM_0](x'-y) \chi_0(x'-z) \mN_1(y-z) S G_1(x-y) G_2(x-z) G_3(y-z) dy dz \Big|^m dx'   dx \\
I_2 & := \int_{B(0,N^{\frac13})}  \int_{B_x}   \Big| \int_{B(0,N^{\frac13})^2}  [\chi_0\mM_0](x'-z) \chi_0(x'-y) \mN_1(y-z) S  G_1(x-y) G_2(x-z) G_3(y-z) dy dz \Big|^m dx'   dx \\
I_3 & := \int_{B(0,N^{\frac13})}  \int_{B_x}   \Big| \int_{B(0,N^{\frac13})^2} R_{y,z}(x') G_1(x-y) G_2(x-z) G_3(y-z)   dy dz \Big|^m dx'   dx \\
\end{align*}
The third term is no problem: denoting again $\langle\xi\rangle  \: := \: \sqrt{1+|\xi|^2}$, 
\begin{equation}  \label{estimate_I_3}
\begin{aligned}
I_3 & \le   C \int_{B(0,N^{\frac13})}     \int_{B_x}  \Big|  \int_{B(0,N^{\frac13})^2} |R_{y,z}(x')|  dy dz \Big|^m dx'   dx \\
& \le C' \int_{B(0,N^{\frac13})} \Big|  \int_{\R^6}\big( \langle x-y \rangle^{-4} + \langle x-z \rangle^{-4} \big) \langle y-z \rangle^{-4}    dy dz \Big|^m  dx \\
& \le C' \int_{B(0,N^{\frac13})} \Big|  \int_{\R^6}\big( \langle y' \rangle^{-4} + \langle z'\rangle^{-4} \big) \langle y'-z' \rangle^{-4}    dy' dz' \Big|^m  dx \le C'' N.
\end{aligned}
\end{equation}
The first and second terms are symmetric in $y$ and $z$, we only consider the first one. As $G_2$ belongs to
 $\R + (L^q \cap L^\infty)$, there exists $C_0 \in \R$ such that  for all $x' \in B_x$, 
$$\chi_0(x'-z) G_2(x-z) =  C_0 + R(x',x,z), \quad \text{with } \:  |R(x',x,z)| \le F(x-z)$$
for $F \in L^q(\R^3) \cap L^\infty(\R^3)$  for some finite $q$. Hence, 
$$I_1 \le C \int_{B(0,N^{\frac13})}(I_a(x) +  I_b(x)) dx$$
 where 
 \begin{align*}
I_a(x) & :=  \int_{B_x}   \Big| \int_{B(0,N^{\frac13})^2} [\chi_0\mM_0](x-y) \chi_0(x'-z) \mN_1(y-z) S G_1(x-y)  G_3(y-z)dy dz \Big|^m dx'   \\
I_b(x) & :=  \int_{B_x}   \Big| \int_{B(0,N^{\frac13})^2} \chi_0(x'-z)  [\chi_0\mM_0](x-y) \mN_1(y-z) S  G_1(x-y)  G_3(y-z)  R(x',x,z)  dy dz \Big|^m dx'  \\
\end{align*}
For the second term, uniformly in $x$, 
\begin{align*}
 I_b(x) 
\le C & \Big| \int_{\R^6}   \frac{1}{\langle x-y\rangle^3}\, \frac{1}{\langle y-z\rangle^3} F(x-z)  dy dz \Big|^m   \\
\le C  & \Big| \int_{\R^6}  \frac{1}{\langle y'\rangle^3 \langle y''\rangle^3} F(y'-y'')  dy' dy''  \Big|^m   = C  \Big|   \int_{\R^3} \frac{1}{\langle y'\rangle^3}  \Big( \frac{1}{\langle\cdot \rangle^3}  \star F \Big)(y') dy'  \Big|^m .  
\end{align*}
As $F$ is in $L^q \cap L^\infty$ for some finite $q$, and  $\frac{1}{\langle \cdot \rangle^3} $ is in $L^p$ for all $p > 1$, the convolution belongs to $L^r$ for all finite $r > q$, and finally the integral is finite. This implies that  
\begin{equation} \label{estim_int_I}
\begin{aligned}
\int_{B(0,N^{\frac13})}  I_b(x) dx  \le  C  N.   
\end{aligned}
\end{equation}
\mspace
As regards $I_a$, we  decompose
\begin{align*}
& I_a(x) \\ 
\le &  C  \int_{B_x}    \Big| \int_{B(0,N^{\frac13})^2}  \big( [\chi_0 \mM_0](x'-y) - [\chi_0\mM_0](x-y)\big)  \mN_1(y-z) S G_1(x-y) G_3(y-z) dy dz \Big|^m dx'   
\\
 +  & C   \Big| \int_{B(0,N^{\frac13})^2} [\chi_0\mM_0](x-y)  \mN_1(y-z) S G_1(x-y) G_3(y-z) dy dz \Big|^m  =: C I_{a,1}(x) + C I_{a,2}(x). 
\end{align*} 
As regards the first term, for all $x' \in B_x$,   $\big|[\chi_0\mM_0](x'-y) - [\chi_0\mM_0](x-y)| \le C \langle x-y\rangle^{-4}$, hence   
\begin{align*}
 I_{a,1}(x)  \le C  \Big| \int_{B(0,N^{\frac13})} \frac{G_1(x-y)}{\langle x-y\rangle^4}  |j_N(y)| dy  \Big|^m \le C'  \Big| \int_{\R^3} \frac{1}{\langle x-y\rangle^4}  |j_N(y)| dy  \Big|^m  , 
\end{align*}
where $j_N :=  \big(\mN_1 S G_3 \big)    \star 1_{B(0,N^{\frac13})}$. The crucial point is that  $  \big(\mN_1 S G_3 \big)  \star$ is continuous over $L^m$. Indeed,  $G_3$  belongs to $\R + L^q \cap L^\infty$. Moreover, $\mN_1$ belongs to $\chi_0 \mM_0 + L^1$.  It follows that  for some $C_0 \in \R$,  $\mN_1 S G_3$ belongs to  $C_0 \chi_0 \mM_0 + L^1$ and therefore is continuous over $L^m$ for any $1 < m < \infty$, see Theorem \ref{CZ_thm}. We deduce 
$$ \|j_N\|_{L^m(\R^3)} \le C \| 1_{B(0,N^{\frac13})}\|_{L^m(\R^3)} \le C' N^{\frac1m} $$
Finally, 
\begin{equation} \label{bound_J1}
\int_{B(0,N^{\frac13})} I_{a,1}(x) dx \le \int_{\R^3}  I_{a,1}(x) dx \le \|  \frac{1}{\langle \cdot \rangle^4}   \star |j_N| \|^m_{L^m} \le C \|j_N\|^m_{L^m(\R^3)} \le C' N 
\end{equation}
Eventually, a similar bound may be established on $I_{a,2}$. Namely, we find that 
\begin{equation} \label{bound_J2}
\begin{aligned}
\int_{B(0,N^{\frac13})} I_{a,2}(x)  dx \le \int_{\R^3}  I_{a,2}(x) dx \le C  \|   [\chi_0 \mM_0]  G_1  \star j_N \|^m_{L^m} \le C \|j_N\|^m_{L^m(\R^3)} \le C' N 
\end{aligned}
\end{equation}
Here, we have used that $\chi_0 \mM_0  G_1  \star$ is continuous over $L^m$, as it belongs to $C_0 \chi_0 \mM_0 + L^1$ for some $C_0$.   We conclude that 
\begin{equation}  \label{estimate_I_1}
I_1 \le C N. 
\end{equation}
As mentioned earlier, by symmetry, the same bound holds for $I_2$: 
\begin{equation}  \label{estimate_I_2}
I_2 \le C N. 
\end{equation}
By gathering estimates \eqref{estimate_I_3}-\eqref{estimate_I_1}-\eqref{estimate_I_2}, we conclude the proof. 
\end{proof}

\begin{corollary} \label{cor_rho_3}
Assume that  $G = G(y,z)$ takes the form : 
\begin{equation} \label{space_G}
\begin{aligned}
& G(y,z) = \sum_{1\le k \le K} G_{1,k}(y) G_{2,k}(z) G_{3,k}(y-z) + R(y,z), \\ 
&\text{where for $1 \le q < \infty$},\\
& G_{i,k} \in \R + (L^q(\R^3) \cap L^\infty(\R^3)) \: \forall i,k, \quad R \in L^q(\R^3 \times \R^3) \cap L^\infty(\R^3 \times \R^3).
\end{aligned}
\end{equation}
Then,  for all $m \in (1,\infty)$, there exists $C > 0$, 
\begin{equation*}
\int_{B(0,N^{\frac13})} \int_{B_x}   \Big| \int_{B(0,N^{\frac13})^2}  S_{y,z}(x')   G(y-x,z-x) dy dz \Big|^m dx'    dx \le C N 
\end{equation*}
\end{corollary}
\begin{proof}
By  the previous Corollary \ref{cor_rho_2}, it remains to bound 
 \begin{align*}
& \int_{B(0,N^{\frac13})} \int_{B_x}   \Big| \int_{B(0,N^{\frac13})^2}  |S_{y,z}(x')|  \, |R(y-x,z-x)| dy dz \Big|^m dx'    dx 
\end{align*}
We use \eqref{crude_bound} to find :  
\begin{align*}  
& \int_{B(0,N^{\frac13})} \int_{B_x}   \Big| \int_{B(0,N^{\frac13})^2}  |S_{y,z}(x')|  |R(y-x,z-x)| dy dz \Big|^m dx'    dx  \\
 \le \: & C \int_{B(0,N^{\frac13})}   \Big| \int_{B(0,N^{\frac13})^2}  \Big(  \frac{1}{\langle x-y\rangle^3} + \frac{1}{\langle x-z\rangle^3} \Big) \frac{1}{\langle y-z\rangle^3} |R(y-x,z-x)|   dy dz \Big|^m dx 
\end{align*}
We find that uniformly in $x$, 
\begin{align*}
 \Big| \int_{B(0,N^{\frac13})^2}    \frac{1}{\langle x-y\rangle^3}  \frac{1}{\langle y-z\rangle^3}   |R(y-x,z-x|  dy dz \Big|   \le &   \:  \int_{\R^6}    \frac{1}{\langle y'\rangle^3}  \frac{1}{\langle y'-z'\rangle^3}  |R(y',z')| dy' dz'  \\
\le &   \:  \Big\|  \frac{1}{\langle \cdot \rangle^{3q'}}    \star  \frac{1}{\langle \cdot \rangle^{3q'} }\Big\|_{L^1} \|R\|_{L^q} \le C  
\end{align*}
and similarly, 
\begin{align*}
  \Big| \int_{B(0,N^{\frac13})^2}   \frac{1}{\langle x-z\rangle^3}  \frac{1}{\langle y-z\rangle^3}  |R(y-x,z-x)|    dy dz \Big|  & \le  C   \end{align*}
 We conclude 
 \begin{equation}
 \int_{B(0,N^{\frac13})} \int_{B_x}   \Big| \int_{B(0,N^{\frac13})^2} | S_{y,z}(x')| \, |R(y-x,z-x)| dy dz \Big|^m dx'  dx  \le C N 
\end{equation}
and the corollary follows. 
 \end{proof}

\mspace
The next corollary can be proved very similarly to the previous one. We state it without proof:  
\begin{corollary} \label{cor_rho_3_bis}
Let $G$ satisfying \eqref{space_G}, and $n \ge m \in (1,\infty)$. There exists $C > 0$, 
\begin{equation*}
\int_{B(0,N^{\frac13})} \int_{B_x} \Big(  \int_{B(0,N^{\frac13})}  \Big| \int_{B(0,N^{\frac13})}  S_{y,z}(x')   G(x,y,z)  dz \Big|^m dy \Big)^{\frac{n}{m}} dx'    dx \le C N 
\end{equation*}
\end{corollary}

\begin{remark} \label{remark_G}
We denote by $\mathcal{G}$ the space of functions $G = G(y,z)$ satisfying \eqref{space_G}, that is the assumptions of Corollaries \ref{cor_rho_3} and  \ref{cor_rho_3_bis}. Note that this space is an algebra, being stable by sum and product. It will play an important role in the proof of our Propositions \ref{prop2} and \ref{prop3}, notably because of the next lemma.
\end{remark}
\begin{lemma} \label{lemma_G}
The functions 
$$(y,z) \rightarrow g_2(0,y), \quad (y,z) \rightarrow g_2(0,z),  \quad (y,z) \rightarrow g_2(y,z)  \quad \text{and}  \quad (y,z) \rightarrow \frac{g_3(0,y,z)}{g_2(0,y)}$$
 belong to $\mathcal{G}$. 
\end{lemma}
\begin{proof}
The first three functions belong to $\mathcal{G}$ as a direct consequence of \eqref{H2}. For the last one, we proceed as follows. 
By \eqref{H2}, for $R$ large enough and $|y| \ge R$, $\rho_2(0,y) \ge \frac{1}{2}$. Let $\chi \in C^\infty(\R^3)$, such that $\chi = 0$ on $B(0,R)$, and 
$\chi =1$ in the large. Combining \eqref{H3} and \eqref{H3'}, we get 
\begin{align*}  
\frac{g_3(0,y,z)}{g_2(0,y)} = (1-\chi(y)) & g_2(0,z)  +  (1-\chi(y))  R_2(y,z)  \\
 + \chi(y) & g_2(0,z)  +  \frac{\chi(y) (g_2(y,z) - 1) g_2(0,z)}{g_2(0,y)}  + \frac{\chi(y)}{g_2(0,y)}  {\tilde R}_2(y,z), \\
 \text{ with } \:  |R_2(y,z)| \le  F(y-z), &\quad  |\tilde{R}_2(y,z)| \le F(y) F(y-z) , \quad F \in L^q \cap L^\infty, \quad 1 \le q <\infty
\end{align*}
The conclusion follows easily. 

\end{proof}

\subsection{Proof of Proposition \ref{main_prop}}
To prove that the limit exists and is bounded with respect to $\phi$, we first notice that from integration by parts, {\it cf.}  \eqref{property_Ii}: 
\begin{align*}
\mI_x(\Psi_{x,y}, P_{x,y}) & =  \int_{\pa B(x,\frac{2+R_0}{4})}   \big(  \sigma(\Psi_{x,y}, P_{x,y})n - 2 \Psi_{x,y} \big)  \cdot Sn.  
\end{align*}
By classical considerations, as the compatibility condition $\int_{\pa B(0,\frac{2+R_0}{4})} Sn = 0$ is satisfied, there exists a smooth matrix valued  $\tilde S = \tilde S(z)$ satisfying 
$$ \div \tilde S = 0 \: \text{ for   $\frac{2+R_0}{4} < |z| < \frac{R_0}{2}$,} \quad  \tilde S\vert_{\pa B(0,\frac{2+R_0}{4})} = S, \quad \tilde 
S\vert_{\pa B(0,\frac{R_0}{2})} = 0.   $$
Similarly, as $\int_{\pa B(0,\frac{2+R_0}{4})} Sn \cdot n = 0$, there exists a smooth vector valued  $\tilde s = \tilde s(z)$ satisfying 
$$ \div \tilde s = 0 \: \text{ for   $\frac{2+R_0}{4} < |z| < \frac{R_0}{2}$,} \quad  \tilde s\vert_{\pa B(0,\frac{2+R_0}{4})} = Sn, \quad \tilde 
s\vert_{\pa B(0,\frac{R_0}{2})} = 0.   $$
We find 
\begin{align*}
\mI_x(\Psi_{x,y}, P_{x,y}) & = 2  \int_{\frac{2+R_0}{4} < |z-x| < \frac{R_0}{2}} D\Psi_{x,y}(z) : \big(-D \tilde s +  \tilde S\big)(z-x) dz. 
\end{align*}
As in $\frac{2+R_0}{4} < |z-x| < \frac{R_0}{2}$, one has both $|x-z|, |z-y| \ge \frac{2+R_0}{4}$.  We can apply Proposition \ref{prop_large} (be careful that the roles of $x$ and $z$ are switched):  
 \begin{align*}
 D \Psi_{x,y}(z)  & =  \Big( \mM_0(z-y) + \mM_0(z-x) \Big)  \mM_l(y-x) S    +  R_{x,y}(z)  \\
& =  \mM_0(z-x)  \mM_0(y-x) S + O(|x-y|^{-4}). 
\end{align*}
By definition \eqref{def_mM0},  the first term reads $D \Phi $, with $\Phi$ satisfying the Stokes system outside $B_x$, together with the boundary condition  $D \Phi = -\mM_0(y-x) S$ at $B_x$. In particular
\begin{align*} 
 & 2  \int_{\frac{2+R_0}{4} < |z-x| < \frac{R_0}{2}} \big( \mM_0(z-x)  \mM_0(y-x) S \big) : \big(-D \tilde s +  \tilde S\big)(z-x) dz \\ 
  = & \mI_x(\Phi, P) =  \frac{20\pi}{3} \mM_0(y-x) S : S, 
 \end{align*}
see \eqref{stress_tensor_Phi_0}-\eqref{stress_tensor_Phi_0_bis}. Eventually, we get that  
\begin{equation*}
\mI_x(\Psi_{x,y}, P_{x,y})  =  \frac{20\pi}{3} \mM_0(y-x) S : S  + O(|x-y|^{-4}), 
\end{equation*}
Let $\mM$ the homogeneous part of degree $-3$ in $\mM_0$. We write 
\begin{align*} 
&   \frac{1}{ 2  |B(0,N^{\frac13})|} \int_{B(0,N^{\frac13})^2}  \mI_x(\Psi_{x,y}, P_{x,y})  g_2(x,y) dx dy  \\
= &  \frac{10\pi}{3 |B(0,N^{\frac13})|}  \int_{B(0,N^{\frac13})^2} [\chi_0 \mM](x-y) S :  S  dx dy  \\
+ &  \frac{1}{ 2  |B(0,N^{\frac13})|} \int_{B(0,N^{\frac13})^2} \chi_0(x-y) \Big( \mI_x(\Psi_{x,y}, P_{x,y})  -  \frac{20\pi}{3}  \mM(x-y) S : S \Big)  dx dy \\
+ & \frac{1}{ 2  |B(0,N^{\frac13})|} \int_{B(0,N^{\frac13})^2} \chi_0(x-y)  \mI_x(\Psi_{x,y}, P_{x,y})  (g_2(x,y) - 1) dx dy \\
 =: & I_N + J_N + K_N
\end{align*} 
with $\chi_0$ the truncation mentioned in Corollary \ref{cor_global_psi_xyz}. 
Clearly,  
$$ |J_N| \le C  \frac{1}{|B(0,N^{\frac13})|} \int_{B(0,N^{\frac13})^2} \frac{1}{\langle x-y \rangle^4} dy dx \le C,$$
and thanks to \eqref{H2}, 
$$ |K_N| \le C  \frac{1}{|B(0,N^{\frac13})|} \int_{B(0,N^{\frac13})^2} \frac{F(x-y)}{\langle x-y \rangle^3} dx dy \le C. $$
Finally, 
\begin{align*}
I_N = &    \frac{10\pi}{ 3 |B(0,N^{\frac13})|}  \int_{B(0,N^{\frac13}-R_0)}\int_{B(0,N^{\frac13})}  [\chi_0 \mM](x-y) S :  S  dx dy + O(N^{-\frac13} \ln N)   \\
= & \frac{10\pi}{ 3 |B(0,N^{\frac13})|} \int_{B(0,N^{\frac13}-R_0)}   \int_{B(0,N^{\frac13})}  \mM(x-y)  S :  S  dx dy \\
+ & \frac{10\pi}{ 3 |B(0,N^{\frac13})|} \int_{B(0,N^{\frac13}-R_0)}   \int_{B(0,N^{\frac13})}  [(\chi_0 -1)\mM](x-y)  S :  S  dx dy +  O(N^{-\frac13} \ln N) 
\end{align*}
where the integrals in $x$ at the right-hand side can be understood in the sense of the principal value, as $\mM(x-y)$ and  $[(\chi_0 -1)\mM]$ have zero average on small  annuli centered at the origin (see Appendix \ref{CZ_thm}). By homogeneity, 
\begin{align*}
& \frac{10\pi}{ 3 |B(0,N^{\frac13})|} \int_{B(0,N^{\frac13}-R_0)}   \int_{B(0,N^{\frac13})}  \mM(x-y)  S :  S  dx dy  \\
= \: &   \frac{10\pi}{ 3 |B(0,1)|} \int_{B(0,1}   \int_{B(0,1)}  \mM(x-y)  S :  S  dx dy + o(1) \\
\end{align*}
while as $\chi_0 - 1$ is supported in $B(0,R_0)$, 
\begin{align*}
&  \frac{10\pi}{ 3 |B(0,N^{\frac13})|} \int_{B(0,N^{\frac13}-R_0)}   \int_{B(0,N^{\frac13})}  [(\chi_0 -1)\mM](x-y)  S :  S  dx dy  \\
& = \frac{10\pi}{3} \int_{B(0,R_0)}  [(\chi_0 -1)\mM](x')  S :  S  dx + o(1)   
\end{align*} 
This concludes the proof of the proposition. 

\begin{remark}
Writing the field $\Phi_0$ in \eqref{def_Phi0} as $\Phi_0(z) = \mathcal{P}_0(z)S$, we notice that our proof of Proposition \ref{main_prop} relies on the formula
$$  D \Psi_{x,y}(z) =    D \Psi^{app}_{x,y}(z) + O(|x-y|^{-4}), $$
with
$$\Psi^{app}_{x,y}(z) :=   \mathcal{P}_0(z-x) \mM(y-x) S,  \quad \text{with} \quad  \mM(x) S :=  D\Big( -\frac{5}{2} S : (x \otimes x) \frac{x}{|x|^5} \Big).$$
 the part homogeneous of degree $-3$ in $\mM_0$. 
One can write 
\begin{equation} \label{formula_mu2_bis}
\begin{aligned}
 \mu_2 & =  \lim_{N \rightarrow +\infty} \frac{1}{ 2 |B(0,N^{\frac13})|} \int_{B(0,N^\frac{1}{3})^2}\mI_x(\Psi^{app}_{x,y}, P^{app}_{x,y}) g_2(x,y) dx dy \\
& + \lim_{N \rightarrow +\infty} \frac{1}{ 2 |B(0,N^{\frac13})|} \int_{B(0,N^\frac{1}{3})^2}\mI_x(\Psi_{x,y} - \Psi^{app}_{x,y},P_{x,y} - P^{app}_{x,y}) g_2(x,y) dx dy
\end{aligned}
\end{equation}
where both limits exist separately, thanks to above arguments. With this decomposition, one can see that the Batchelor-Green formula in Theorem \ref{main_thm} is also valid for processes satisfying \eqref{Hstrong}, and studied previously in \cite{GVH}. Indeed, we proved in \cite{GVH} that under \eqref{Hstrong}, the second order correction is given by 
 $$ \nu_2 S : S :=  \lim_{N \rightarrow +\infty} \frac{1}{ 2 |B(0,N^{\frac13})|} \sum_{i \neq k \in I_N}  \mI_i(\Psi^{app}_{x_i, x_k}, P^{app}_{x_i,x_k})  $$
which in the random setting coincides with the limit of the expectations, that is 
$$ \nu_2 S : S =  \lim_{N \rightarrow +\infty} \frac{1}{ 2 |B(0,N^{\frac13})|} \int_{B(0,N^\frac{1}{3})^2}\mI_x(\Psi^{app}_{x,y}, P^{app}_{x,y}) g_2(x,y) dx dy.  $$
Moreover, under \eqref{Hstrong}, the second term at the r.h.s. of \eqref{formula_mu2_bis} satisfies 
\begin{align*}
  \limsup_{N \rightarrow +\infty} & \frac{1}{ 2 |B(0,N^{\frac13})|} \int_{B(0,N^\frac{1}{3})^2}\mI_x(\Psi_{x,y} - \Psi^{app}_{x,y},P_{x,y} - P^{app}_{x,y}) g_2(x,y) dx dy \\
 \le & \frac{C}{|B(0,N^{\frac13})|} \int_{B(0,N^\frac{1}{3})^2 \cap \{|x-y| \ge c \phi^{-\frac13}\}} \frac{1}{|x-y|^{4}} dy dx \le  C  \phi^{1/3}
\end{align*}
Hence,  $\nu_2 = \mu_2 + O(\phi^{1/3})$ showing that the Batchelor-Green formula also applies to the setting considered in \cite{GVH}.
\end{remark}

\subsection{Proof of Theorem \ref{main_thm}}  \label{subsec_proof_thm}
Inspired by the cluster expansion  \eqref{approx_u_NS}, we define $u_{N,err}$ ({\em the error term}) through: 
\begin{equation} \label{exact_u_NS}
\begin{aligned}
 u_{I_N,S} & = u_{\emptyset, S} + \sum_{\{k\} \subset I_N  } (u_{\{k\}, S} - u_{\emptyset, S}) +  \sum_{\substack{\{k,l\} \subset I_N,\\ k \neq l }} \big(  u_{\{k,l\}, S} - u_{\{k\},S}  - u_{\{l\},S} + u_{\emptyset, S} \big) + u_{N,err}  \\
 & =  Sx  + \sum_{\{k\} \subset I_N  } \Phi_{\{k\}}  +  \sum_{\substack{\{k,l\} \subset I_N,\\ k \neq l }}  \Psi_{\{k,l\}} + u_{N,err} 
\end{aligned}
\end{equation}
where functions $u_I$, $\Phi_I$ and $\Psi_{k,l}$ were introduced in \eqref{stokes}, \eqref{def_Phi_I} and \eqref{def_Psi_kl} respectively. 
By Proposition  \ref{approx_effective_viscosity}, following the formal calculations of Paragraph \ref{subsec_formal}, we find that 
\begin{align*}
 \mu_h S : S & = |S|^2 + \frac{5}{2}\phi |S|^2  \\ 
 & +  \lim_{N \rightarrow +\infty} \bigg( \frac{1}{ 2  |B(0,N^{\frac13})|} \int_{B(0,N^{\frac13})^2}  \mI_x(\Psi_{x,y}, P_{x,y})  \rho_2(x,y) dx dy   \\
 & +  \E  \frac{1}{ 2 |B(0,N^{\frac13})|} \sum_{i \in I_N}  \mI_i(u_{N,err}, p_{N,err})   \bigg)
\end{align*}
where $\mI_i$, $\mI_x$, $\Psi_{x,y}$ were introduced in \eqref{def_I_i} and \eqref{def_Psi_I_xy}. 

\mspace
Key estimates are provided by: 

\begin{proposition} \label{prop2}
$$ \limsup_{N \rightarrow +\infty} \E \frac{1}{|B(0,N^{\frac13})|} \int_{\cup_{i \in I_N} B_i} |D(u_{N,err})|^2 = O(\phi^3) $$\end{proposition}

\begin{proposition} \label{prop3}
$$ \limsup_{N \rightarrow +\infty} \E \frac{1}{|B(0,N^{\frac13})|} \int_{\cup_{i \in I_N} B_i} |D (\Phi_{I_N} - \sum_{k \in I_N} \Phi_{\{k\}}) |^2 = O(\phi^2)  $$\end{proposition}

\mspace
Let us show how they imply Theorem \ref{main_thm}.
we have to show that
\begin{equation}
\limsup_{N \rightarrow +\infty}  \Big|    \E  \frac{1}{|B(0,N^{\frac13})|} \sum_{i \in I_N}  \mI_i(u_{N,err}, p_{N,err}) \Big|  =    O(\phi^{\frac52}) 
\end{equation}
We write 
\begin{align*}
\sum_{i \in I_N}  \mI_i(u_{N,err}, p_{N,err}) & = \sum_{i \in I_N} \int_{\pa B_i} \sigma(u_{N,err}, p_{N,err})n \cdot Sn - 2 \sum_{i \in I_N} \int_{\pa B_i}  u_{N,err}  \cdot Sn \\ 
& = - \sum_{i \in I_N} \int_{\pa B_i} \sigma(u_{N,err}, p_{N,err})n \cdot \Phi_{I_N,S} - 2 \sum_{i \in I_N} \int_{B_i} D(u_{N,err}) : S
\end{align*}
Here, we have used the fact that $\Phi_{I_N,S} = -S(x-x_i) + \text{rigid vector field}$,  and the integral relations in \eqref{stokes}. Note that $u_{N,err}$ and  $\Phi_{I_N,S}$ both solve a homogeneous Stokes equation outside $\cup B_i$. Hence, after a double integration by parts, we find 
\begin{align*}
   \sum_{i \in I_N}  \mI_i(u_{N,err}, p_{N,err}) &  =   - \sum_{i \in I_N}  \int_{\pa B_i} u_{N,err} \cdot \sigma(\Phi_{I_N}, P_{I_N}) n - 2 \sum_{i \in I_N} \int_{B_i} D(u_{N,err}) : S \\
 &  =  -\sum_{i \in I_N}  \int_{\pa B_i} u_{N,err} \cdot \sigma\big(\Phi_{I_N} - \sum_k \Phi_{\{k\}}, P_{I_N} - \sum_k P_{\{k\}}\big) n \\
 &   \quad   -\sum_{i \in I_N}  \int_{\pa B_i} u_{N,err}   \cdot \sigma\big(\sum_k \Phi_{\{k\}} , \sum_k P_{\{k\}}\big) n  - 2 \sum_{i \in I_N} \int_{B_i} D(u_{N,err}) : S  \\
 &  =: \: A_N + B_N 
 \end{align*}
The first term can be written 
\begin{align*}
|A_N| & =  \Big| 2 \int_{\R^3\setminus (\cup B_i)}  D(u_{N,err})  : D\big(\Phi_{I_N} - \sum_k \Phi_{\{k\}}\big) \Big| \\
& \le 2 \Big( \int_{\R^3}  |D(u_{N,err})|^2 \int_{\R^3}  |D (\Phi_{I_N} - \sum_{k \in I_N} \Phi_{\{k\}}) |^2 \Big)^{1/2} \\
& \le C \Big( \int_{\cup_{i \in I_N} B_i}  |D(u_{N,err})|^2 \int_{\cup_{i \in I_N} B_i}  |D (\Phi_{I_N} - \sum_{k \in I_N} \Phi_{\{k\}}) |^2 \Big)^{1/2} 
\end{align*}
where the last line comes from the well-known minimizing properties of solutions of homogeneous Stokes equations with prescribed symmetric gradients at the boundary of the domain.  Using Cauchy-Schwarz inequality, Propositions \ref{prop2} and \ref{prop3}, we find that 
$$ \limsup_{N} \E  \frac{1}{|B(0,N^{\frac13})|}  |A_N| = O(\phi^\frac52). $$
As regards $B_N$, we split the sum over $k$ in two: 
\begin{align*}
B_N =  - &\sum_{i \in I_N}   \int_{\pa B_i} u_{N,err}   \cdot \sigma\big(\Phi_{\{i\}} , P_{\{i\}}\big) n  \: - \: \sum_{i \in I_N}  \int_{\pa B_i} u_{N,err}   \cdot  \sigma\big(\sum_{k \neq i} \Phi_{\{k\}} , \sum_{k \neq i} P_{\{k\}}\big) n \\
 -2 &\sum_{i \in I_N} \int_{B_i} D(u_{N,err}) : S 
\end{align*}
As  $\sum_{k \neq i} \Phi_{\{k\}}$ solves the homogeneous equation inside $B_i$, we find
\begin{align*}
B_N & =  -\sum_{i \in I_N}   \int_{\pa B_i} u_{N,err}   \cdot \sigma\big(\Phi_{\{i\}} , P_{\{i\}}\big) n  \: - \:  \sum_{i \in I_N}  \int_{B_i} D(u_{N,err}) : D(\sum_{k \neq i} \Phi_{\{k\}} ) \\
&  -2 \sum_{i \in I_N} \int_{B_i} D(u_{N,err}) : S \\
& =  -\sum_{i \in I_N}   \int_{\pa B_i} u_{N,err}   \cdot \sigma\big(\Phi_{\{i\}} , P_{\{i\}}\big) n  \: - \:  \sum_{i \in I_N}  \int_{B_i} D(u_{N,err}) : D(\sum_{k} \Phi_{\{k\}} - \Phi_{I_N}) \\
& -2 \sum_{i \in I_N} \int_{B_i} D(u_{N,err}) : S \\
& =  -\sum_{i \in I_N}   \int_{\pa B_i} u_{N,err}   \cdot \sigma\big(\Phi_{\{i\}} , P_{\{i\}}\big) n  -2 \sum_{i \in I_N} \int_{B_i} D(u_{N,err}) : S + O(\phi^{\frac52}) N
\end{align*}
using Cauchy-Schwarz and Propositions \ref{prop2} and \ref{prop3} to bound the second term. As regards the first term,   thanks to \eqref{stress_tensor_Phi_0},
$$    \int_{\pa B_i} u_{N,err}   \cdot \sigma\big(\Phi_{\{i\}} , P_{\{i\}}\big) n -2 \sum_{i \in I_N} \int_{B_i} D(u_{N,err}) : S  = \int_{B_i} D(u_{N,err}) : S  $$
Moreover, by definition of $u_{N,err}$, one has 
\begin{equation} \label{trace_u_err}
D(u_{N,err})\vert_{B_i} = -  \sum_{\substack{\{k,l\} \subset I_N \setminus \{i\}, \\ k \neq l}}  D(\Psi_{\{k,l\}}) 
\end{equation}
Hence, 
\begin{align*}
B_N =   \frac{1}{2} \sum_{\substack{i,k,l  \in I_N \\ i \neq k \neq l}}  \int_{\pa B_i}  D(\Psi_{\{k,l\}}) : S    + O(\phi^{\frac52}) N
\end{align*}
so that 
\begin{align*} 
|\E B_N| & =  \Big| \frac{1}{2}  \int_{B(0,N^{\frac13})} \int_{B_x}   \Big( \int_{B(0,N^{\frac13})^2}  D \Psi_{y,z}(x') \rho_3(x,y,z) dy dz \Big)  : S dx' dx  \Big|   +   o(\phi^2 N)  \\
& \le  C N^{1/2}  \Big(  \int_{B(0,N^{\frac13})} \int_{B_x}   \Big| \int_{B(0,N^{\frac13})^2}  D \Psi_{y,z}(x')  \rho_3(x,y,z) dy dz \Big|^2  dx' dx \Big)^{1/2}   \\
& \le  C \phi^3 N 
\end{align*}
where the second bound comes from Cauchy-Schwarz inequality, and the third one comes from Corollary \ref{cor_rho_3}. It follows that 
$$ \limsup_{N} \E  \frac{1}{|B(0,N^{\frac13})|}  |B_N| = O(\phi^{\frac52}). $$
This concludes the proof of Theorem \ref{main_thm}. 
\subsection{Proof of auxiliary Proposition \ref{prop2}}
By \eqref{trace_u_err}, $\E \int_{\cup_{i \in I_N} B_i} |D(u_{N,err})|^2$ is fully explicit in terms of $\Psi_{y,z}$ and the first correlation functions of the process. We find 
\begin{align*}
 & \E \int_{\cup_{i \in I_N} B_i} |D(u_{N,err})|^2  =  \sum_{i \in I_N} \E \int_{B_i}   \Big| \sum_{\substack{\{k,l\} \subset I_N \setminus \{i\}  \\ k \neq l}}  D(\Psi_{\{k,l\}}) \Big|^2 \\
= & \frac{1}{2} \int_{B(0,N^{\frac13})^3}   \int_{B_x}  |   D \Psi_{y,z}(x') |^2   \rho_3(x,y,z) dy dz dx' dx \\
+ & \frac{1}{4} \int_{B(0,N^{\frac13})^4}   \int_{B_x}   D \Psi_{y,z}(x') : D(\Psi_{y,z'})(x')  \rho_4(x,y,z,z') dy dz dz' dx' dx \\
+ & \frac{1}{4}  \int_{B(0,N^{\frac13})^5}   \int_{B_x}   D \Psi_{y,z}(x') : D(\Psi_{y',z'})(x')  \rho_5(x,y,z,y',z') dy dy' dz dz' dx' dx \\
:= & \frac{1}{2} \phi^3 I + \frac{1}{4} \phi^4 J +  \frac{1}{4} \phi^5 K. 
\end{align*}
Using \eqref{crude_bound} we find 
\begin{align*} 
|I| & = \int_{B(0,N^{\frac13})^3}   \int_{B_x}  | S_{y,z}(x') |^2   g_3(x,y,z) dy dz dx' dx \\
& \le C   \int_{B(0,N^{\frac13})^3} \Big(  \frac{1}{\langle x-y\rangle^6} +  \frac{1}{\langle x-z\rangle^6} \Big) \frac{1}{\langle y-z\rangle^6}  dx dy dz \\
& \le  C N.  
\end{align*}
For the analysis of $J$, we rely on \eqref{H4}. Defining $G(y,z) :=    \frac{g_3(0,y,z)}{\sqrt{g_2(0,y)}}$, we find 
\begin{align*}
 g_4(0,y,z,z') &  = g_3(0,y,z)    \Big( \frac{g_3(0,y,z')}{g_2(0,y)}  + R_3(y,z,z')\Big), \\
 & = G(y,z) G(y,z') +  R(y,z,z'),  \quad |R(y,z,z')| \le F(z-z'), \quad F \in L^q \cap L^\infty. 
 \end{align*}
Hence,
\begin{align*} 
|J| & \le   \int_{B(0,N^{\frac13})}   \int_{B_x}   \int_{B(0,N^{\frac13})}  \Big| \int_{B(0,N^{\frac13})}  S_{y,z}(x')  G(y-x,z-x) dz\Big|^2 dy   dx' dx \\
& +   \int_{B(0,N^{\frac13})^4}  \int_{B_x}   \Big| \int_{B(0,N^{\frac13})}  |S_{y,z}(x')| \,    |S_{y,z'}(x')|  \, F(z-z') dx' dx dy dz dz'
\end{align*}
By Remark \ref{remark_G} and Lemma \ref{lemma_G},  $G$ belongs to $\mathcal{G}$, meaning it satisfies the assumptions \eqref{space_G} of Corollaries \ref{cor_rho_3} and \ref{cor_rho_3_bis}. 
The first term is controlled thanks to Corollary \ref{cor_rho_3_bis}, with  $n=m=2$. For the other term, we use again \eqref{crude_bound}: 
\begin{align*} 
& \int_{B(0,N^{\frac13})^4}  \int_{B_x}   \Big| \int_{B(0,N^{\frac13})}  |S_{y,z}(x')| \,    |S_{y,z'}(x')|  \, F(z-z') dx' dx dy dz dz' \\ 
\le & \int_{B(0,N^{\frac13})^4} \Big(  \frac{1}{\langle x-y \rangle^3} +   \frac{1}{\langle x-z \rangle^3}  \Big)   \frac{1}{\langle y-z \rangle^3} 
 \Big(  \frac{1}{\langle x-y \rangle^3} +   \frac{1}{\langle x-z' \rangle^3}  \Big)   \frac{1}{\langle y-z' \rangle^3} F(z-z')  dx dy dz dz' 
 \end{align*}
Expanding, one can see that this term is bounded by $C N$. First, 
\begin{align*}
& \int_{B(0,N^{\frac13})^4}   \frac{1}{\langle x-y \rangle^6}   \frac{1}{\langle y-z \rangle^3}   \frac{1}{\langle y-z' \rangle^3} F(z-z')  dx dy dz dz' \\
\le & C N \int_{\R^9}   \frac{1}{\langle Y \rangle^6}   \frac{1}{\langle Y-Z \rangle^3}   \frac{1}{\langle Y-Z' \rangle^3} F(Z-Z') dY dZ dZ'  \\ 
\le & C N  \int_{\R^9}   \frac{1}{\langle Y \rangle^6}   \frac{1}{\langle s \rangle^3}   \frac{1}{\langle s' \rangle^3} F(s-s')  ds ds' dY
\le  C' N \Big\| \Big( F \star \langle  \cdot \rangle^{-3} \Big)   \langle  \cdot \rangle^{-3}  \Big\|_{L^1} \le C'' N
\end{align*}
Also, 
\begin{align*}
& \int_{B(0,N^{\frac13})^4}   \frac{1}{\langle x-y \rangle^3}   \frac{1}{\langle y-z \rangle^3}   \frac{1}{\langle x-z' \rangle^3}   \frac{1}{\langle y-z' \rangle^3} F(z-z') dx dy dz dz' \\
\le & C N \int_{\R^9} \frac{1}{\langle Y \rangle^3}    \frac{1}{\langle Y-Z \rangle^3}  \frac{1}{\langle Z' \rangle^3}   \frac{1}{\langle Y-Z' \rangle^3} F(Z-Z') dY dZ dZ' \\
\le &  C N   \int_{\R^9}  \frac{1}{\langle Y \rangle^3}  \frac{1}{\langle s \rangle^3}  \frac{1}{\langle s'+Y \rangle^3}     \frac{1}{\langle s' \rangle^3}  F(s-s') dY ds ds'  \\
\le & C N   \int_{\R^9}  \frac{1}{\langle Y \rangle^3}  \frac{1}{\langle s \rangle^3}  \frac{1}{\langle s'+Y \rangle^3}  F(s-s') dY ds ds'  \\
\le  & C N \int_{\R^3}  \frac{1}{\langle Y \rangle^3}  \Big(  \langle  \cdot \rangle^{-3}   \star \big( F(-\cdot) \star  \langle  \cdot \rangle^{-3}  \big) \Big)(Y) dY \le C' N
\end{align*} 
Indeed, arguing as we did several times before, one can show that  
$  \langle  \cdot \rangle^{-3}   \star \big( F(-\cdot) \star  \langle  \cdot \rangle^{-3}  \big) $ belongs to $L^q$ for any $q > r$, so that the product with $\langle  \cdot \rangle^{-3}$ is an integrable function of $Y$. 
Eventually 
\begin{align*}
& \int_{B(0,N^{\frac13})^4}   \frac{1}{\langle x-z \rangle^3}   \frac{1}{\langle y-z \rangle^3}   \frac{1}{\langle x-z' \rangle^3}   \frac{1}{\langle y-z' \rangle^3} F(z-z') dx dy dz dz' \\
\le & C N \int_{\R^9} \frac{1}{\langle Z \rangle^3}    \frac{1}{\langle Y-Z \rangle^3}  \frac{1}{\langle Z' \rangle^3}   \frac{1}{\langle Y-Z' \rangle^3} F(Z-Z') dY dZ dZ' \\
\le &  C N   \int_{\R^9}  \frac{1}{\langle Z \rangle^3}  \frac{1}{\langle t \rangle^3}  \frac{1}{\langle Z' \rangle^3}   \frac{1}{\langle t + Z-Z' \rangle^3}    F(Z-Z') dt dZ dZ' \\
=   & C N \int_{\R^3}   \frac{1}{\langle Z \rangle^3}   \frac{1}{\langle Z' \rangle^3}     \Big(  \langle  \cdot \rangle^{-3} \star   \langle  \cdot \rangle^{-3} \Big)(Z-Z')  F(Z-Z')  dZ dZ'  \le C' N
\end{align*} 
Here, we used that the function  $\langle  \cdot \rangle^{-3} \star   \langle  \cdot \rangle^{-3}$ belongs to $L^q$ for any $q > 1$, so that $\tilde{F} := \Big( \langle  \cdot \rangle^{-3} \star   \langle  \cdot \rangle^{-3} \Big) F$ belongs to $L^1$. Hence, $\langle  \cdot \rangle^{-3} \star \tilde{F}$ belongs to $L^q$ for any $q > 1$, and its product with $\langle  \cdot \rangle^{-3}$ is integrable.

\mspace
The last step is to handle $K$. Note that the quantity  $D \Psi_{y,z}(x') : D(\Psi_{y',z'})(x')$ is invariant with respect to the permutations  $(y \: z)$, $(y' \: z')$ and $(y \: y') (z \:  z')$ of the set $\{y,z,y',z'\}$. We use the following result on the structure of $g_5$: 
\begin{lemma} \label{lemma_g5}
The expression $g_5(0,y,y'z,z')$ can be written as a sum of terms of four types: 
\begin{align*}
g_{\rm I}(y,y',z,z') & = G(y,z) G'(y',z')   \\
g_{\rm II}(y,y',z,z') & = G(y,z) G(y',z')  R_{\rm II}(\sigma y, \sigma y') \\  
g_{\rm III}(y,y',z,z') & = G(y,z) G(y',z')  R_{\rm III}(\sigma y,  \sigma z, \sigma z') \\
g_{\rm IV}(y,y',z,z') & = G(y,z) G'(y',z')  R_{\rm IV}(\sigma y, \sigma y', \sigma z, \sigma z')
\end{align*}
where  $G$, $G'$ change from line to line and  belong to the space $\mathcal{G}$ (see Remark \eqref{remark_G}), where $\sigma$ is a  permutation of   $\{y,z,y',z'\}$ generated by  $(y \: z)$, $(y' \: z')$ and $(y \: y') (z \:  z')$, and where for some $F \in  L^q \cap L^\infty$, with $q < \infty$: 
\begin{equation} \label{remaindersI-IV}
\begin{aligned}
|R_{\rm II}(y,y')|  & \le F(y-y')  \\
|R_{\rm III}(y,y',z)| &  \le F(y-y')    \Big( F(z-y') + F(z-y) \Big) +   F(z-y') \Big( F(y) + F(y') \Big), \\ 
|R_{\rm IV}(y,y',z,z')|  & \le F(y-y')  \Big( F(z-z') + F(y'-z) F(z') + F(z-y) F(z') + F(z) F(z') \Big) 
\end{aligned}
\end{equation} 
\end{lemma}
This lemma will be proved in Appendix \ref{appendix_C}. Pondering on this decomposition, it is enough to control terms of type $K_{\rm I}$ to $K_{\rm IV}$, obtained from $K$ by replacing $g_5$ by a function of the form $g_{\rm I}$ to $g_{\rm IV}$. 
As regards $K_{\rm I}$, we can simply apply Corollary  \ref{cor_rho_3}. Indeed, due to the structure of $g_{\rm I}$, we can bound $K_{I}$ by a sum of terms of the form 
\begin{align*}
& C \Big|  \int_{B(0,N^{\frac13})^5}   \int_{B_x}   S_{y,z}(x')   G(y-x,z-x) : S_{y,z}(x') G'(y'-x,z'-x)  dy dy' dz dz' dx' dx \Big|   
\end{align*}
with  $G,G' \in \mathcal{G}$. 
Applying  Corollary  \ref{cor_rho_3} with $m=2$, 
\begin{align*}
 & \Big|  \int_{B(0,N^{\frac13})^5}   \int_{B_x}   S_{y,z}(x')   G(y-x,z-x)  : S_{y,z}(x')  G'(y'-x,z'-x)   dy dy' dz dz' dx' dx \Big|  \\
\le & \frac{1}{2}    \int_{B(0,N^{\frac13})}  \int_{B_x}  \Big|  \int_{B(0,N^{\frac13})^2}   S_{y,z}(x')   G(y-x,z-x) dy dz  \Big|^2 dx' dx \\
+  &  \frac{1}{2} \int_{B(0,N^{\frac13})}  \int_{B_x}  \Big|  \int_{B(0,N^{\frac13})^2}   S_{y',z'}(x')  G'(y'-x,z'-x)  dy' dz'  \Big|^2 dx' dx 
\le C N. 
\end{align*}
resulting in 
$$ K_{\rm I} \le C N. $$
As regards $K_{\rm IV}$, using the invariance by permutation mentioned above, it is enough to handle the case $\sigma = I_{\rm d}$,  that  is $g_{\rm IV}(y,y',z,z')  = R_{\rm IV}(y, y',z,z')$, where $R_{\rm IV}$ obeys the bound in \eqref{remaindersI-IV}. 
For the sake of brevity, we restrict to the treatment of the first term at the  r.h.s. The other ones can be treated with very similar arguments and are shown  to be $o(N)$. Hence, we shall bound
\begin{align*} 
K'_{\rm IV} & :=   \int_{B(0,N^{\frac13})^5} \int_{B_x}  |S_{y,z}(x')|  \,   |S_{y',z'}(x')| F(y-y') F(z-z') dy dy' dz dz'  dx  \\
\end{align*}
Using  inequality \eqref{crude_bound}, one checks easily  that the function 
$$  M_x : (y,z) \rightarrow   |S_{y,z}(x)|  $$
satisfies 
\begin{equation} \label{integrability_Psi}
 \forall r > 1, \quad \exists C_r > 0,  \quad \sup_{x \in \R^3} \| M_{x} \|_{L^r(\R^6)} \le C_r.
 \end{equation}
 Back to $K'_{\rm IV}$, we notice that 
\begin{align*} 
|K'_{\rm IV}| & \le  C  N \sup_{x \in \R^3} \big\| M_{x} \big( G \star M_{x} \big) \big\|_{L^1(\R^6)} 
 \end{align*}
 where $G$ is the function of the couple $(y,z)$ defined by $G(y,z) = F(y) F(z)$. From \eqref{integrability_Psi} and the fact that $G$ belongs to $L^q(\R^6)$ as $F$ belongs to $L^q(\R^3)$, we deduce that  the $\sup$ at the right-hand side is finite, so that eventually. 
 $$  |K'_{\rm IV}|  \le  C  N, \quad \text{and} \quad  |K_{\rm IV}|  \le  C  N  $$
 It remains to treat $K_{\rm II}$ and $K_{\rm III}$. Again, one can restrict to the case where  $\sigma =  I_{\rm d}$. 
 A keypoint is played here by Corollary \ref{cor_rho_3_bis}, which states that for any $G \in \mathcal{G}$, 
$$ S_G(x',x,y)  := \int_{B(0,N^{\frac13})}  S_{y,z}(x') G(y-x,z-x) dz $$
satisfies for all $n \ge m > 1$, 
\begin{equation} \label{estimate_Sk_1}
  \int_{B(0,N^{\frac13})} \int_{B_x}   \|1_{B(0,N^{\frac13})} S_G(x',x,\cdot)\|_{L^m(\R^3)}^n dx' dx  \le C N 
\end{equation}
Due to the structure of the term $R_{II}$,  depending only on $y,y'$, we find that $K_{\rm II}$ is bounded by terms of the form
\begin{align*}
& C \int_{B(0,N^{\frac13})} \int_{B_x}  \int_{B(0,N^{\frac13})^2}    \big| S_G(x',x,y) \big| \:   \big| S_{G'}(x',x,y') \big| F(y-y') dy dy' dx' dx \\
\le & C \int_{B(0,N^{\frac13})}  \int_{B_x}  \|1_{B(0,N^{\frac13})} \, S_G(x',x,\cdot)\|_{L^{m}(\R^3)}  \| 1_{B(0,N^{\frac13})}  S_{G'}(x',x,\cdot) \star  F \|_{L^{m'}(\R^3)} dx' dx   \\
\le & C \int_{B(0,N^{\frac13})}  \int_{B_x} \|1_{B(0,N^{\frac13})} \,  S_G(x',x,\cdot)\|_{L^{m}(\R^3)}  \| 1_{B(0,N^{\frac13})}   S_{G'}(x',x,\cdot)\|_{L^s(\R^3)} \|F \|_{L^q(\R^3)}  dx' dx \\
\le & C' \int_{B(0,N^{\frac13})}  \int_{B_x} \|1_{B(0,N^{\frac13})} \, S_G(x',x,\cdot)\|_{L^{m}(\R^3)}  \| 1_{B(0,N^{\frac13})} S_{G'}(x',x,\cdot)\|_{L^s(\R^3)}dx' dx
\end{align*}  
for all  $m > 1$ and $s$ such that $\frac{1}{s} + \frac{1}{q} = 1 + \frac{1}{m'}$. As $s < m'$, we can pick an index $n$ such that $n \ge m$ and $n' \ge s$. By H\"older inequality with exponents $n$ and $n'$, and by inequality \eqref{estimate_Sk_1}, we conclude that
\begin{align*}
& \int_{B(0,N^{\frac13})} \int_{B_x}  \int_{B(0,N^{\frac13})^2}    \big| S_G(x',x,y) \big| \:   \big| S_{G'}(x',x,y')  \big| F(y-y') dy dy' dx' dx 
\le   CN 
\end{align*} 
so that eventually 
$$ K_{\rm II} \le C N. $$
As regards $K_{\rm III}$,  it is bounded by  
\begin{align*}
& C  \int_{B(0,N^{\frac13})} \int_{B_x}  \int_{B(0,N^{\frac13})^3}    |S_{y,z}(x')|  \,      \big| S_{G'}(x',x,y')\big| \, F(y-y') F(z-y')      dy dy' dz dx' dx  \\
+ & \: C \int_{B(0,N^{\frac13})} \int_{B_x}  \int_{B(0,N^{\frac13})^3}    |S_{y,z}(x')|  \,     \big|S_{G'}(x',x,y')\big| \, F(y-y') F(z-y)      dy dy' dz dx' dx \\
+ & \: C \int_{B(0,N^{\frac13})} \int_{B_x}  \int_{B(0,N^{\frac13})^3}    |S_{y,z}(x')|  \,    \big|S_{G'}(x',x,y')\big| \, F(z-y') F(y)      dy dy' dz dx' dx \\
+  & \: C \int_{B(0,N^{\frac13})} \int_{B_x}  \int_{B(0,N^{\frac13})^3}    |S_{y,z}(x')|  \,    \big|S_{G'}(x',x,y')\big| \, F(z-y') F(y')      dy dy' dz dx' dx \\
 =: & K_{\rm III}'+  K_{\rm III}'' +  K_{\rm III}''' + K_{\rm III}''''
\end{align*}
We shall focus on the first one for brevity, the others are very similar. We find 
\begin{align*}
K_{\rm III}' \le & 
C'  \int_{B(0,N^{\frac13})} \int_{B_x}  \int_{B(0,N^{\frac13})^3}    \Big( \frac{1}{\langle x'-y \rangle^3} +   \frac{1}{\langle x'-z\rangle^3} \Big)  \\
& \hspace{3cm}  \frac{1}{\langle y-z\rangle^{3}}  \big| S_{G'}(x',x,y')\big| \, F(y-y') F(z-y')      dy dy' dz dx' dx \\
\le & C'  \int_{B(0,N^{\frac13})} \int_{B_x}  \int_{B(0,N^{\frac13})}   \big|S_{G'}(x',x,y')\big|  \, F_N(x',y') dy' dx' dx 
\end{align*}
 where 
 \begin{align*}
 F_N(x',y') & :=   \int_{B(0,N^{\frac13})^2}   \Big( \frac{1}{\langle x'-y \rangle^3} +   \frac{1}{\langle x'-z\rangle^3} \Big)    \frac{1}{\langle y-z\rangle^{3}} F(y-y') F(z-y')      dy  dz \\
 & \le \int_{\R^6}   \Big( \frac{1}{\langle x'-y \rangle^3} +   \frac{1}{\langle x'-z\rangle^3} \Big)    \frac{1}{\langle y-z\rangle^{3}} F(y-y') F(z-y')      dy  dz\\
 & \le \int_{\R^6}   \Big( \frac{1}{\langle x - y' - Y \rangle^3} +   \frac{1}{\langle x - y' - Z\rangle^3} \Big)   \frac{1}{\langle Y-Z\rangle^{3}} F(Y) F(Z)   dY  dZ \\
 & \le  2 \, \tilde{F}(x-y'), \quad  \tilde{F} := \Big(F F \star \langle \cdot \rangle^{-3} \Big) \star  \langle \cdot \rangle^{-3}     \in L^r(\R^3) \: \forall r > q.  
\end{align*}
Choosing $m > 1$ such that $m' = r$, we deduce
\begin{align*}
C & \int_{B(0,N^{\frac13})} \int_{B_x}  \int_{B(0,N^{\frac13})^3}    |S_{y,z}(x')|  \,  \,    \big| S_{G'}(x',x,y')\big| \, F(y-y') F(z-y')      dy dy' dz dx' dx \\
\le &  C'  \left(  \int_{B(0,N^{\frac13})} \int_{B_x}  \int_{B(0,N^{\frac13})}   \big| S_{G'}(x',x,y')\big|^m dy' dx' dx \right)^{\frac{1}{m}} 
 \left(  \int_{B(0,N^{\frac13})} \int_{B_x}  \int_{B(0,N^{\frac13})}   \big| \tilde{F}(x'-y')\big|^r dy' dx' dx \right)^{\frac{1}{r}}  \\
\le & C'' N^{\frac{1}{m}} N^{\frac{1}{r}} = C'' N. 
\end{align*}
Eventually, we find 
$$ K_{\rm I} + K_{\rm II} +   K_{\rm III} +   K_{\rm IV}  \le C N $$
which concludes the proof of the Proposition. 

\subsection{Proof of auxiliary Proposition \ref{prop3}}
 We only sketch the proof, as it is much simpler than the one of Proposition \ref{prop2}, although in the same spirit. We want a $O(\phi^2 N)$ bound on  
 \begin{align*}
&  \E  \int_{\cup_{i \in I_N} B_i} |D (\Phi_{I_N} - \sum_{k \in I_N} \Phi_{\{k\}}) |^2   =    \E  \sum_{i \in I_N}  \int_{B_i} |\sum_{k \in I_N, k \neq i} D \Phi_{\{k\}}) |^2 
\\
 = & \E  \sum_{k \neq i \in I_N}   \int_{B_i}    |D \Phi_{\{k\}}|^2 \: + \: \E  \sum_{k \neq k' \neq i \in I_N}   \int_{B_i}    D \Phi_{\{k\}} : D \Phi_{\{k'\}} \\
 = & \phi^2  \int_{B(0,N^{\frac13})} \int_{B_x}  \int_{B(0,N^{\frac13})} |D \Phi_y(x')|^2 g_2(x,y) dy dx' dx  \\ 
 + & \phi^3  \int_{B(0,N^{\frac13})} \int_{B_x}  \int_{B(0,N^{\frac13})^2} D \Phi_y(x') : D \Phi_{y'}(x') g_3(x,y,y') dy dy' dx' dx \: =: \: \phi^2 I_1 +\phi^ 3 I_2.    
 \end{align*} 
  We remind that the one-sphere solution $\Phi_y$ satisfies \eqref{corrector_y}, and that $D \Phi_y(x) = \mM_0(x-y) S$, {\it cf.} \eqref{def_mM0}. 
 The first term is simply bounded by 
 $$ |I_1| \le C N \sup_{x'}  \int_{\R^3} \frac{1}{\langle x'-y\rangle^6} dy \le C N. $$
 For the second term, we use   \eqref{H3} to write (remind $g_2(x,y) = g(x-y))$:
 $$g_3(x,y,y') = g(x-y) g(x-y') + R_2(x,y,y'), \quad  |R_2(x,y,y')| \le F(y-y'), \quad F \in L^q \cap L^\infty.$$
 Hence, 
 \begin{align*}
 I_2 & \le C   \int_{B(0,N^{\frac13})} \int_{B_x} \Big|   \int_{B(0,N^{\frac13})} \chi_0(x'-y) \mM_0(x'-y) S g(x-y) dy \Big|^2  dx' dx \\
 & + C   \int_{B(0,N^{\frac13})} \int_{B_x}  \int_{B(0,N^{\frac13})^2} |D \Phi_y(x')|  \, |D \Phi_{y'}(x')| F(y-y') dy dy' dx' dx := I_{2,a} + I_{2,b}
 \end{align*}
 Uniformly in $x$, we find that 
\begin{align*}
&  \int_{B_x}  \int_{B(0,N^{\frac13})^2} |D \Phi_y(x')|  \, |D \Phi_{y'}(x')| F(y-y') dy dy' dx' 
 \le & C  \int_{\R^6}  \frac{1}{\langle y \rangle^3 \langle y' \rangle^3 }F(y-y') dy dy' < +\infty
 \end{align*}
 so that $I_{2,b} \le C N$.  As regards $I_{2,a}$, we decompose $\chi_0 \mM_0(x'-y)  = (\chi_0 \mM_0(x'-y) - \chi_0 \mM_0(x-y)) +  \chi_0 \mM_0(x-y)$   resulting in 
 \begin{align*}
I_{2,a}& \le C   \int_{\R^3} \Big|   \int_{B(0,N^{\frac13})} \frac{1}{\langle x-y \rangle^4} dy  \Big|^2 dx  +  C  \int_{\R^3} \Big|   \int_{B(0,N^{\frac13})} \chi_0 \mM_0(x-y) dy  \Big|^2 dx \\
& \le  C \| \langle  \rangle^{-4} \star 1_{B(0,N^{\frac13})} \|_{L^2}^2 + C \| \mM_0 \star 1_{B(0,N^{\frac13})} \|_{L^2}^2  \le C N, 
 \end{align*}
 using that convolution with $\langle  \rangle^{-4}$ or  $\chi_0 \mM_0$ is continuous over $L^2$. This ends the proof. 
 
 \section*{Acknowledgements}
 The author acknowledges the support of  the Institut Universitaire de France, and of the French National Research Agency (ANR) through the SingFlows project, grant ANR-18- CE40-0027.

\appendix
\section{Proof of Proposition \ref{approx_effective_viscosity}} \label{appA}
We remind that  
$$\mu_h S : S = \E |D \phi_S + S|^2 = |S|^2 +   \E |D \phi_S|^2$$
 as $\E D \Phi_S = 0$. We admit temporarily that 
\begin{equation} \label{bigcell_limit}
\E |D \phi_S|^2 =  \lim_{N \rightarrow +\infty} \E  \frac{1}{|B(0,N^{\frac13})|} \int_{\R^3} |D u_{I_N,S} - S|^2 
\end{equation}
Admitting this claim, 
\begin{align*} 
 \int_{\R^3} |D u_{I_N,S} - S|^2  & =  \int_{\cup B_i} |D u_{I_N,S} - S|^2 +  \int_{\R^3 \setminus \cup B_i} |D u_{I_N,S} - S|^2 \\ 
 & = |\cup B_i| \, |S|^2 -  \frac{1}{2} \int_{\cup \pa B_i}  \sigma( u_{I_N,S} - Sx, p_{I_{N,S}})n \cdot (u_{I_N,S} - Sx)   \\
 & =  |\cup B_i| \, |S|^2 - \frac{1}{2}  \int_{\cup \pa B_i}  \sigma(u_{I_N,S}, p_{I_N,s})n \cdot  u_{I_N,S}  +   \int_{\cup \pa B_i} S n \cdot u_{I_N,S}  \\
 & + \frac{1}{2}\int_{\cup \pa B_i}    \sigma(u_{I_N,S}, p_{I_N,s})n \cdot Sx -  \int_{\cup \pa B_i} S n \cdot Sx 
 \end{align*}
As $u_{I_N, S}$ is a rigid vector field in $B_i$, the second term is zero by the fourth and fifth lines of \eqref{stokes}, and the third term is zero as well.  The first and last terms compensate, and we are left with 
\begin{align*} 
 \int_{\R^3} |D u_{I_N,S} - S|^2  & =   \frac{1}{2}\int_{\cup \pa B_i}    \sigma(u_{I_N,S}, p_{I_N,s})n \cdot Sn \\
 & =  \frac{1}{2} \sum_{i \in I_N}  \int_{\pa B_i} \big( \sigma(u_{I_N,S}, p_{I_N,S})n - 2 u_{I_N,S}  \big)   \cdot Sn 
 \end{align*}
 where we added artificially the term $u_{I_N,S} \cdot Sn$ which has zero integral at the boundary, again because $u_{I_N,S}$ is rigid. Back to the expression of $\mu_h S : S$, we find the formula stated in  Proposition \ref{approx_effective_viscosity}. 
 
 \mspace
 It remains to understand formula  \eqref{bigcell_limit}. This kind of formula is now classical in homogenization theory (see for instance \cite{MR2044813}): one recovers the homogenized matrix, given by the corrector problem set in $\R^3$, through approximations of this cell problem on larger and larger finite domains (here $B(0,N^{\frac13})$).  The conditions set at the boundary do not really matter: one can use periodicity conditions, Dirichlet conditions, or like here, extension by a homogeneous Stokes solution outside $B(0,N^{\frac13})$. The case of Dirichlet conditions is covered in \cite{DuerinckxGloria}. For the setting considered here, we may rely on the previous work \cite{GVH}, where a full treatment is given for a kind of point approximation of system \eqref{corrector}. We only give below the  sketch of the proof, and  refer to \cite[Propositions 5.2 and 5.3]{GVH} for details. 
 
\mspace
Let $\eps := N^{-\frac13}$, $v_{\eps}(x) := \eps u_{I_N,S}\big(\frac{x}{\eps}\big) -  Sx$, $B_{i,\eps} :=  B(\eps x_i, \eps) \subset B(0,1)$. One has 
\begin{equation} \label{stokes_veps}
\begin{aligned}
 -\Delta v_\eps + \na p_\eps = 0, \quad \div v_\eps & = 0   \quad \text{in } \R^3 \setminus (\cup B_{i, \eps}) \\
 D(v_\eps) + S & = 0 \quad \text{in } \cup B_{i, \eps}, \\
 \int_{\pa B_{i,\eps}} \sigma(v_\eps,p_\eps)n & = \int_{\pa B_{i,\eps}} \sigma(v_\eps,p_\eps)n  \times x = 0 \quad \forall i
\end{aligned}
\end{equation}
Moreover, formula  \eqref{bigcell_limit} is equivalent to 
\begin{equation}
 \E |D \phi_S|^2 = \lim_{\eps \rightarrow 0} \frac{1}{|B(0,1)|} \E  \int_{\R^3}  |D v_\eps|^2.
\end{equation} 
Under assumption \eqref{H1}, it is standard to show that the solution $v_\eps$ of  \eqref{stokes_veps} is bounded in $\dot{H}^1$ uniformly in the realization of the point process. Hence, by the dominated convergence theorem, it is enough to prove that 
$$ \E |D \phi_S|^2 = \lim_{\eps \rightarrow 0} \frac{1}{|B(0,1)|}  \int_{\R^3}  |D v_\eps|^2 \quad \text{almost surely}. $$
To prove this property, one  introduces an approximation $(\overline{v}_{\eps}, \overline{p}_\eps$ defined by the following conditions:  $\overline{v}_{\eps} \in \dot{H}^1(\R^3)$, 
\begin{align*}
\overline{v}_{\eps}(x) & = \eps \Phi_S\big(\frac{x}{\eps}\big) - \dashint_{B(0,1)} \eps \Phi_S\big(\frac{\cdot}{\eps}\big) , \quad x \in B(0,1) \\
\overline{p}_{\eps}(x) &= P_S\big(\frac{x}{\eps}\big) - \dashint_{B(0,1)}  P_S\big(\frac{\cdot}{\eps}\big) , \quad x \in B(0,1) \\
-\Delta \overline{v}_{\eps} + \na \overline{p}_\eps & = 0, \quad \div \overline{v}_{\eps}  = 0 \quad \text{ outside } B(0,1). 
\end{align*}
The idea is then to show that 
$$ \E |D \phi_S|^2 = \lim_{\eps \rightarrow 0}   \frac{1}{|B(0,1)|} \int_{\R^3}  |D \overline{v}_\eps|^2 \quad \text{almost surely} $$
and that 
$$ \lim_{\eps \rightarrow 0}  \int_{\R^3}  |D (v_\eps - \overline{v}_\eps)|^2 = 0.$$
The first, resp. second property, is the analogue of Proposition 2, resp. Proposition 3 in \cite{GVH} (compare $v_\eps$ and $\overline{v}_\eps$ to $\eps^3 h_\eps^\eta$ and $\eps^3 \overline{h}_\eps^\eta$ in \cite{GVH}). Adapting these propositions to the current setting requires minor changes, that are left to the reader.   
\section{Calderon-Zygmund theorem} \label{appendix_CZ}
We use a version of the Calderon-Zygmund theorem given in  \cite{Met}.

\begin{theorem} \label{CZ_thm}
 Let $f \in C^1(\R^n\setminus\{0\})$ satisfying  
\begin{itemize}
\item[i)] $\forall |\alpha| \le 1,  \exists C_\alpha > 0,  \quad |\pa_\alpha f(x)| \le \frac{C_\alpha}{|x|^{n+|\alpha|}} \quad \forall x \neq 0$.
\item[ii)] There exists $ C > 0$ such that for all $0 < t < t' < +\infty$, $\: \Big| \int_{t < |x| < t'} f \Big| \le C$.
\end{itemize}
Let $\theta \in C^\infty_c(\R^n)$ that is $1$ near the origin, and let $T_\theta$  the distribution defined by 
$$  \langle T_\theta , \varphi \rangle = \int_{R^n}  f(x) \left(  \varphi(x) - \varphi(0) \theta(x) \right) dx, \quad \forall \varphi \in C^\infty_c(\R^n).   $$
Then, the operator $\varphi \rightarrow T_\theta \star \varphi$ extends as a continuous operator from $L^p(\R^n)$ to itself for all $1 < p < \infty$.
 \end{theorem}
In particular, it follows that for $\chi_0$ smooth, $\chi_0 = 0$  near the origin, $\chi_0 = 1$ far away, and for $\mathcal{M}_0$ given in \eqref{def_mM0}, the convolution operator $\varphi \rightarrow [\chi_0 \mathcal{M}_0] \star \varphi$ sends continuously  $L^p(\R^n)$ to itself for all $1 < p < \infty$ : indeed, it satisfies i) and ii), and coincides with $T_\theta$ if $\theta$ is chosen with support in the region where $\chi_0$ vanishes. The fact that it satisfies ii) is deduced easily from the property 
$$ \forall 0 < t < t' <\infty, \quad  \int_{t < |x| < t'} \mM = 0, $$
where $\mM$ is the homogeneous part of degree $-3$ in $\mM_0$.  This last identity comes from Green's formula, as  $\mM$ can be written   as a combination of derivatives of functions homogeneous of degree $-2$.

\section{Proof of Lemma \ref{lemma_g5}} \label{appendix_C}
The basic idea behind the lemma is to have a splitting of $g_5$ into factors that depend only either on $(y,z)$ or on $(y',z')$, and belong to the space $\mathcal{G}$, see  \eqref{space_G}. This is of course not fully possible because of the remainders  in  assumptions \eqref{H2}-\eqref{H5} or \eqref{H3'}-\eqref{H5'}, responsible for terms of type II to IV.   We start from the decomposition of $g_5$ given by \eqref{H5'}. The remainder in this decomposition is of type IV. As regards the third term, we simply write :
\begin{align*}
& \frac{g_3(0,y,z)  g_3(0,y',z)  g_3(0,y,z')  g_3(0,y',z')}{g_2(0,y) g_2(0,y') g_2(0,z) g_2(0,z')} = \frac{g_3(0,y,z)}{g_2(0,y)} \frac{g_3(0,y',z')}{g_2(0,y')} \frac{g_3(0,y',z)}{g_2(0,z)} \frac{g_3(0,y,z') }{g_2(0,z')}   \\
= &  \frac{g_3(0,y,z)}{g_2(0,y)} \frac{g_3(0,y',z')}{g_2(0,y')}  \big( g_2(0,y') + R_2(y',z) \big)    \big( g_2(0,y) + R_2(y,z') \big)   
\end{align*}
where the last equality comes from the first line of \eqref{H3}. By Lemma \ref{lemma_G}, expanding this last expression yields a sum of terms of type I, II and IV.  
Back to \eqref{H5'}, it remains to handle the first two terms at the right-hand side. By symmetry, it is enough to handle the second term. There exists $R > 0$ such that for $|z| \ge R$, $z' \ge R$ and $|z-z'| \ge R$, $g_3(0,z,z') \ge \frac{1}{2}$. Let $\chi$ a smooth function with values in $[0,1]$ such that $\chi = 0$ on  $B(0,R)$ and $\chi=1$ in the large. We split 
\begin{align*}
\frac{g_4(0,y,z,z') g_4(0,y',z,z')}{g_3(0,z,z')}  & \: = \:  \big( 1 - \chi(z) \chi(z') \chi(z-z')  \big)  \frac{g_4(0,y,z,z') g_4(0,y',z,z')}{g_3(0,z,z')} \\ 
 & \: + \:  \chi(z) \chi(z') \chi(z-z') \frac{g_4(0,y,z,z') g_4(0,y',z,z')}{g_3(0,z,z')}  =: A_{1-\chi} + A_\chi. 
\end{align*}
To handle $A_{1-\chi}$, it is enough to use decompositions given by  \eqref{H3}-\eqref{H4}. Hence, 
\begin{align*}
& \frac{g_4(0,y,z,z') g_4(0,y',z,z')}{g_3(0,z,z')}   =      \Big( \frac{g_3(0,y,z) g_3(0,z,z')}{g_2(0,z)} +  R_3(z,y,z') \Big) \Big( \frac{g_3(0,y',z') }{g_2(0,z')} + R_3(z',y',z) \Big)  \\
  =  &    \Big( g_3(0,y,z)  \big( g_2(0,z') + R_2(z,z')) +  R_3(z,y,z') \Big)  \Big( \frac{g_3(0,y',z') }{g_2(0,z')} + R_3(z',y',z) \Big)  
 \end{align*}
where 
$$ | R_2(z,z'))| \le F(z-z'), \quad |R_3(z,y,z')| \le F(y-z'). $$
By expanding, we find that $\frac{g_4(0,y,z,z') g_4(0,y',z,z')}{g_3(0,z,z')}$ is a sum of terms of type I,  type  II (with factor $R_2(z,z')$), type III (that is $R_2(z,z') R_3(z',y',z)$), type IV  (that is $R_3(y',y,z')  R_3(z',y',z)$), as well as two terms of the form 
$$ G(y,z)  G'(y',z')  R_3(z,y,z')  \quad \text{ and } \quad   G(y,z)  G'(y',z')  R_3(z',y',z), \quad G, G_2 \in \mathcal{G}. $$
 Multiplying by $\big( 1 - \chi(z) \chi(z') \chi(z-z')  \big)$ the terms of type I to IV give terms of the same kind, while the last two terms give terms of type III, taking into account that 
 $$  \big( 1 - \chi(z) \chi(z') \chi(z-z')  \big) \le F(z) + F(z') + F(z-z') $$
 for some $F$ compactly supported. Hence,  $A_{1-\chi}$ has the right structure. 
 
 \mspace
It remains to treat$A_{\chi}$.  By \eqref{H3}, 
$$ g_3(0,z,z') =  g_2(0,z) \big( g_2(0,z') + R_2(z,z') \big), \quad |R_2(z,z')| \le F_2(z-z').  $$
By taking $R$ large enough, we can further assume that $\frac{R_2(z,z')}{g_2(0,z')} \le \frac{1}{2}$ over the support of 
 $\chi(z) \chi(z') \chi(z-z')$. Finally, 
 \begin{equation} \label{decompo_chi}
 \begin{aligned}
  \frac{\chi(z) \chi(z') \chi(z-z')}{g_3(0,z,z')} & = \frac{\chi(z)}{g_2(0,z)} \frac{\chi(z')}{g_2(0,z')} \chi(z-z')(1 + R'_2(z,z')),   \quad |R'_2(z,z')| \le F(z-z')
  \end{aligned}
  \end{equation}
As regards the other factor in $A_\chi$, that is $g_4(0,y,z,z') g_4(0,y',z,z')$, we use \eqref{H4'}, then \eqref{H3}-\eqref{H3'}, then \eqref{H2} to separate as much as possible  terms in $(y,z)$ from terms in $(y',z')$. For instance,
\begin{align*}
& g_4(0,y,z,z')  \\
= &   \frac{g_3(0,y,z) g_3(0,y,z')}{g_2(0,y)} +  \frac{g_3(0,z,z') g_3(y,z,z')}{g_2(z,z')}  - g_2(0,z) g_2(y,z) g_2(0,z') g_2(y,z')  +  \tilde{R}_3(y,z,z') \\
 = & g_3(0,y,z) \big(g_2(0,z') + R_2(y,z') \big) \\
 + & \frac{1}{g_2(z,z')}  g_2(0,z) \big(g_2(0,z') + R_2(z,z') \big) \big(  g_2(y,z) g_2(y,z') +  (g_2(z,z') - 1) g_2(y,z')  + \tilde{R}_2(z-y,z'-y) \big) \\
 - & g_2(0,z) g_2(y,z) g_2(0,z')  g_2(y,z') + \tilde{R}_3(y,z,z') 
 \end{align*}
 and eventually 
 \begin{align*}
 g_4(0,y,z,z')   = &  g_3(0,y,z) \big(g_2(0,z') + R_2(y,z') \big) \\
 + & \frac{1}{g_2(z,z')}  g_2(0,z)  \big(g_2(0,z') + R_2(z,z') \big) \\
 &  \times \Big(  (g_2(y,z) (1 + R_1(y-z')) +  R_1(z-z') (1+R_1(y-z'))   + \tilde{R}_2(z-y,z'-y) \Big) \\
 - & g_2(0,z) g_2(y,z) g_2(0,z')  (1+R_1(y-z')) + \tilde{R}_3(y,z,z') 
 \end{align*}
Let us note that $\frac{1}{g_2(z,z')} = 1 + R'_1(z-z')$ with $R'_1 \in L^q \cap L^\infty$ on the support of $\chi(z-z')$. Using the symmetric decomposition for  $g_4(0,y',z,z')$, together with decomposition \eqref{decompo_chi}, we end up with a decomposition of $A_\chi$ which is seen through a tedious but straightforward calculation to have the right structure.

{\footnotesize
\bibliographystyle{siam}

 \bibliography{biblio_Batchelor_Green}
 }
 \end{document}